\documentclass[a4paper,11pt]{article}
\usepackage{amscd,amsthm,amsfonts,amscd,epsf,amsmath,amssymb,latexsym,multicol}
\usepackage[latin1]{inputenc}
\usepackage{pstricks,pst-node,cite,epsfig}
\usepackage{graphicx}
\usepackage{xcolor}
\usepackage{mathrsfs}
\usepackage{textcomp}
\usepackage{setspace}
\usepackage{array}
\usepackage{float}
\usepackage{fancyhdr}
\usepackage{url}
\usepackage{pgf,tikz}
\usetikzlibrary{arrows,automata,positioning,decorations.pathreplacing,decorations.markings}
\usepackage[margin=2.5cm]{geometry}
\usepackage{cite,tabularx,
caption,fancyhdr,color,indentfirst}

\begin{document}

\theoremstyle{definition}
\newtheorem{dfn}{Definition}
\newtheorem{thm}{Theorem}
\newtheorem{lem}[thm]{Lemma}
\newtheorem {pro}[thm]{Proposition}
\newtheorem{cor}[thm]{Corollary}
\newtheorem{rmk}[thm]{Remark}

\title{\textsc{\Huge Cross-connections of completely simple semigroups}\footnote{Some of the results of this paper was presented at the International Conference in Algebra and its Applications, Aligarh Muslim University, India held in December 2014.}}
\author{Azeef Muhammed P. A. \footnote{The author wishes to acknowledge the financial support (via JRF and SRF) of the Council for Scientific and Industrial Research, New Delhi, India in the preparation of this article.}\\ \small Department Of Mathematics, University Of Kerala, \\ \small Trivandrum, Kerala-695581, India. \\ \small \emph{azeefp@gmail.com} \and A. R. Rajan \footnote{The author acknowledges the financial support of the Kerala State Council for Science, Technology and Environment, Trivandrum (via the award of Emeritus scientist) in the preparation of this article. } \\ \small Department Of Mathematics, University Of Kerala, \\ \small Trivandrum, Kerala-695581, India. \\ \small \emph{arrunivker@yahoo.com } }

\date{September, 2015 }
\maketitle

\begin{abstract}
A completely simple semigroup $S$ is a semigroup without zero which has no proper ideals and contains a primitive idempotent. It is known that $S$ is a regular semigroup and any completely simple semigroup is isomorphic to the Rees matrix semigroup $\mathscr{M}[G;I,\Lambda;P]$ (cf. \cite{rees}).\\ 
In the study of structure theory of regular semigroups, K.S.S. Nambooripad introduced the concept of normal categories to construct the semigroup from its principal left(right) ideals using cross-connections. A normal category $\mathcal{C}$ is a small category with subobjects wherein each object of the category has an associated idempotent normal cone and each morphism admits a normal factorization. A cross-connection between two normal categories $\mathcal{C}$ and $\mathcal{D}$ is a \emph{local isomorphism} $\Gamma : \mathcal{D} \to N^\ast\mathcal{C}$ where $N^\ast\mathcal{C}$ is the normal dual of the category $\mathcal{C}$.\\
In this paper, we identify the normal categories associated with a completely simple semigroup $S = \mathscr{M}[G;I,\Lambda;P]$ and show that the semigroup of normal cones $T\mathcal{L}(S)$ is isomorphic to a semi-direct product $G^\Lambda \ltimes \Lambda$. We characterize the cross-connections in this case and show that each sandwich matrix $P$ correspond to a cross-connection. Further we use each of these cross-connections to give a representation of the completely simple semigroup as a cross-connection semigroup.\\[.2cm]
\end{abstract}
\section{Introduction}
In the study of the structure theory of regular semigroups, T.E. Hall (cf. \cite{hall}) used the ideal structure of the regular semigroup to analyse its structure. P.A. Grillet (cf. \cite{gril}) refined Hall's theory to abstractly characterize the ideals as \emph{regular partially ordered sets} and constructing the fundamental image of the regular semigroup as a cross-connection semigroup. K.S.S. Nambooripad (cf. \cite{cross}) generalized the idea to any arbitrary regular semigroups by characterizing the ideals as \emph{normal categories}.  A normal category $\mathcal{C}$ is a small category with subobjects wherein each object of the category has an associated idempotent normal cone and each morphism admits a normal factorization. All the normal cones in a normal category form a regular semigroup $T\mathcal{C}$ known as the semigroup of normal cones. The principal left(right) ideals of a regular semigroup $S$ with partial right(left) translations as morphisms form a normal category $\mathcal{L}(S)$. A cross-connection between two normal categories $\mathcal{C}$ and $\mathcal{D}$ is a \emph{local isomorphism} $\Gamma : \mathcal{D} \to N^\ast\mathcal{C}$ where $N^\ast\mathcal{C}$ is the normal dual of the category $\mathcal{C}$. A cross-conection $\Gamma$ determines a cross-connection semigroup $\tilde{S}\Gamma$ and conversely every regular semigroup is isomorphic to a cross-connection semigroup
for a suitable cross-connection.\\
A completely simple semigroup is a regular semigroup without zero which has no proper ideals and contains a primitive idempotent. It is known that any completely simple semigroup is isomorphic to the Rees matrix semigroup $\mathscr{M}[G;I,\Lambda;P]$ where $G$ is a group $I$ and $\Lambda$ are sets and $P = (p_{\lambda i})$ is a $\Lambda \times I$ matrix with entries in $G$(cf. \cite{rees}). Then $\mathscr{M}[G;I,\Lambda;P] = G \times I \times \Lambda$ and the binary operation is given by
$$(a,i,\lambda)(b,j,\mu) = (ap_{\lambda j}b,i,\mu) .$$
In this paper, we characterize the normal categories involved in the construction of a completely simple semigroup as a cross-connection semigroup. We show that the category of principal left ideals of $S$ -  $\mathcal{L}(S)$ has $\Lambda$ as its set of objects and $G$ as the set of morphisms between any two objects. We observe that it forms a normal category and we characterize the semigroup of normal cones arising from this normal category. We show that this semigroup is equal to the semi-direct product of $G^{\Lambda} \ltimes \Lambda$. We characterize the principal cones in this category and show that the principal cones form a regular sub-semigroup of $G^\Lambda \ltimes \Lambda$. We also show for $\gamma_1 =(\bar{\gamma_1},\lambda_k), \gamma_2= (\bar{\gamma_2},\lambda_l) \in T\mathcal{L}(S)$, $\gamma_1 \mathscr{L} \gamma_2$ if and only if $\lambda_k = \lambda_l$ and $\gamma_1 \mathscr{R} \gamma_2$ if and only if $\bar{\gamma_1}   G = \bar{\gamma_2}   G$. Further we characterize the cross-conncetions in this case and show that different cross-connections correspond to different sandwich matrices $P$. We also use these cross-connections to give a representation of the completely simple semigroup $S$ as a cross-connection semigroup $S\Gamma$ wherein each element of the semigroup is represented as a normal cone pair $(\gamma,\delta)$ and multiplication of elements is a semi-direct product multiplication.
\section{Preliminaries} 
In the sequel, we assume familiarity with the definitions and elementary concepts of category theory (cf. \cite{mac}). The definitions and results on cross-connections are as in \cite{cross}. For a category $\mathcal{C}$, we denote by $v\mathcal{C}$ the set of objects of $\mathcal{C}$. 
\begin{dfn}
A \emph{preorder} $\mathcal{P}$ is a category such that for any $ p , p' \in v\mathcal{P} $, the hom-set $\mathcal{P}(p,p')$ contains atmost one morphism.
\end{dfn}
In this case the relation $\subseteq$ on the class $v\mathcal{P}$ of objects of $\mathcal{P}$ defined by
$ p\subseteq p'\iff \mathcal{P}(p,p')\ne\emptyset $ is a quasi-order. $\mathcal{P}$ is said to be a strict preorder if $\subseteq$ is a partial order.
\begin{dfn}
Let $\mathcal{C}$ be a category and $\mathcal{P}$ be a subcategory of $\mathcal{C}$. Then $(\mathcal{C} ,\mathcal{P})$ is called a \emph{category with subobjects} if $\mathcal{P}$ is a strict preorder with $v\mathcal{P} = v\mathcal{C}$ such that every $f\in \mathcal{P}$ is a monomorphism in $\mathcal{C}$ and if $f,g\in \mathcal{P}$ and if $f=hg$ for some $h \in \mathcal{C}$, then $h\in \mathcal{P}$.
In a category with subobjects, if $f: c \to d $ is a morphism in $ \mathcal{P}$, then $f$ is said to be an \emph{inclusion}. And we denote this inclusion by $j(c,d)$.
\end{dfn}
\begin{dfn}
A morphism $e: d \to c$ is called a
\emph{retraction} if $c \subseteq d$ and $j(c,d) e = 1_{c}$.
\end{dfn}
\begin{dfn}
A \emph{normal factorization} of a morphism $f \in \mathcal{C}(c,d)$ is a
factorization of the form $f=euj$ where $e:c\to c'$ is a retraction,
$u:c'\to d'$ is an isomorphism and $j=j(d',d)$ for some $c',d' \in v\mathcal{C}$ with $ c' \subseteq c$, $ d' \subseteq d$. 
\end{dfn}
\begin{dfn}\label{dfn1}
 Let $d\in v\mathcal{C}$. A map $\gamma:v\mathcal{C}\to\mathcal{C}$ is called a \emph{cone from the base $v\mathcal{C}$ to the vertex $d$} if $\gamma(c)\in \mathcal{C}(c,d)$ for all $c\in v\mathcal{C}$ and whenever $c'\subseteq c$ then $j(c',c)\gamma(c) = \gamma(c')$. The cone $\gamma$ is said to
   be \emph{normal} if there exists $c\in v\mathcal{C}$ such that $\gamma(c):c\to c_{\gamma}$ is an isomorphism.
\end{dfn}
Given the cone $\gamma$ we denote by $c_{\gamma}$ the the \emph{vertex} of
$\gamma$ and for each $c\in v\mathcal{C}$, the morphism $\gamma(c): c \to c_{\gamma}$
is called the \emph{component} of $\gamma$ at $c$. We define $ M\gamma = \{ c \in \mathcal{C}\: :\: \gamma(c)\text{ is an isomorphism} \}$.
\begin{dfn}
A \emph{normal category} is a pair $(\mathcal{C}, \mathcal{P})$ satisfying the following :
\begin{enumerate}
\item $(\mathcal{C}, \mathcal{P})$ is a category with subobjects.
\item Any morphism in $\mathcal{C}$ has a normal factorization. 
\item For each $c \in v\mathcal{C} $ there is a normal cone $\sigma$ with vertex $c$ and $\sigma (c) = 1_c$.
\end{enumerate}
\end{dfn}
\begin{thm}(cf. \cite{cross} )
 Let $(\mathcal{C}, \mathcal{P})$ be a normal category and let $T\mathcal{C}$ be the set of all normal cones in $\mathcal{C}$. Then $T\mathcal{C}$ is a regular semigroup with product defined as follows :\\
 For $\gamma, \sigma \in T\mathcal{C}$.
\begin{equation} \label{eqnsg}
(\gamma * \sigma)(a) = \gamma(a) (\sigma(c_\gamma))^\circ
\end{equation} 
 where $(\sigma(c_\gamma))^\circ$ is the epimorphic part of the $\sigma(c_\gamma)$. Then it can be seen that $\gamma * \sigma$ is a normal cone. $T\mathcal{C}$ is called the \emph{semigroup of normal cones} in $\mathcal{C}$.
 \end{thm}
For each $\gamma \in T\mathcal{C}$, define $H({\gamma};-)$ on the objects and morphisms of $\mathcal{C}$ as follows. For each $c\in v\mathcal{C}$ and for each $g\in \mathcal{C}(c,d)$, define
 \begin{subequations} \label{eqnH}
   \begin{align}
     H({\gamma};{c})&= \{\gamma\ast f^\circ : f \in \mathcal{C}(c_{\gamma},c)\} \\
     H({\gamma};{g}) &:\gamma\ast f^\circ \mapsto \gamma\ast (fg)^\circ
   \end{align}
 \end{subequations}

\begin{dfn}
If $\mathcal{C}$ is a normal category, then the \emph{normal dual} of $\mathcal{C}$ , denoted by $N^\ast \mathcal{C}$ , is the full subcategory of $\mathcal{C}^\ast $ with vertex set\\
\begin{equation} \label{eqnH1}
v N^\ast \mathcal{C} = \{ H(\epsilon;-) \: : \: \epsilon \in E(T\mathcal{C}) \}
\end{equation} 
where $\mathcal{C}^\ast$ is the category of all functors from $\mathcal{C}$ to $\bf{Set}$(cf. \cite{mac} ). 
\end{dfn}
\begin{thm} (cf. \cite{cross} )\label{thm1}
Let $\mathcal{C}$ be a normal category. Then $N^\ast \mathcal{C}$ is a normal category and is isomorphic to $\mathcal{R}(T\mathcal{C})$ as normal categories.
For $\gamma, \gamma' \in T\mathcal{C}$ , $H(\gamma;-) = H(\gamma';-)$ if and only if there is a unique isomorphism $h: c_{\gamma'} \to c_\gamma$, such that $\gamma = \gamma' \ast h $. And $\gamma \mathscr{R} \gamma' \: \iff \: H(\gamma;-) = H(\gamma';-)$.\\
To every morphism $\sigma : H(\epsilon;-) \to H(\epsilon ';-)$ in $N^\ast \mathcal{C} $, there is a unique $\hat{\sigma} : c_{\epsilon '} \to c_\epsilon$ in $\mathcal{C}$ such that the component of the natural transformation $\sigma$ at $c \in v\mathcal{C} $ is the map given by :
\begin{equation} \label{eqnH2}
\sigma(c) : \epsilon \ast f^\circ \mapsto \epsilon ' \ast (\hat{\sigma} f)^\circ 
\end{equation}
Also $H({\gamma};-) = H({\gamma}';-)$ implies $M\gamma = M{\gamma'}$ where $ M\gamma = \{ c \in \mathcal{C}\: :\: \gamma(c)\text{ is an isomorphism} \}$; and hence we will denote $M\gamma $ by $MH(\gamma;-)$. 
\end{thm}
\begin{dfn}\label{ideal}
Let $\mathcal{C}$ be a small category with subobjects. Then an \emph{ideal} $\langle c \rangle$ of $\mathcal{C}$ is the full subcategory of $\mathcal{C}$ whose objects are subobjects of $c$ in $\mathcal{C}$. It is called the principal ideal generated by $c$.
\end{dfn}
\begin{dfn} \label{lociso}
Let $\mathcal{C}$ and $\mathcal{D}$ be normal categories. Then a functor $F: \mathcal{C} \to \mathcal{D}$ is said to be a \emph{local isomorphism} if $F$ is inclusion preserving , fully faithful and for each $c \in v\mathcal{C}$, $F_{|\langle c \rangle}$ is an isomorphism of the ideal $\langle c \rangle$ onto $\langle F(c) \rangle$.
\end{dfn}
\begin{dfn} \label{cxn}
Let $\mathcal{C}$ and $\mathcal{D}$ be normal categories. A \emph{cross-connection} is a triplet $(\mathcal{D},\mathcal{C};{\Gamma})$ where $\Gamma: \mathcal{D} \to N^\ast\mathcal{C}$ is a local isomorphism such that for every $c \in v\mathcal{C}$, there is some $d \in v\mathcal{D}$ such that $c \in M\Gamma(d)$.
\end{dfn}
\begin{rmk}\label{dualcxn}
Given a cross-conection $\Gamma$ from $\mathcal{D}$ to $\mathcal{C}$, there is a \emph{dual cross-connection} $\Delta$ from $\mathcal{C} $ to $\mathcal{D}$ denoted by ($\mathcal{C}$,$\mathcal{D}$;$\Delta$).
\end{rmk}
\begin{rmk}\label{bif}
Given a cross-connection $\Gamma: \mathcal{D} \to N^\ast\mathcal{C}$, we get a unique bifunctor $\Gamma(-,-) : \mathcal{C}\times\mathcal{D} \to \bf{Set}$ defined by $ \Gamma(c,d) = \Gamma(d)(c) $ and $ \Gamma(f,g) = (\Gamma(d)(f))(\Gamma(g)(c')) = (\Gamma(g)(c))(\Gamma(d')(f))$ for all $(c,d) \in v\mathcal{C}\times v\mathcal{D}$ and $(f,g):(c,d) \to (c',d')$. Similarly corresponding to $\Delta: \mathcal{C} \to N^\ast\mathcal{D}$, we have $\Delta(-,-) : \mathcal{C}\times\mathcal{D} \to \bf{Set}$ defined by $\Delta(c,d) = \Delta(c)(d)$ and $\Delta(f,g) = (\Delta(c)(g))(\Delta(f)(d')) = (\Delta(f)(d))(\Delta(c')(g)) $ for all $(c,d) \in v\mathcal{C}\times v\mathcal{D}$ and $(f,g):(c,d) \to (c',d')$.
\end{rmk}
\begin{thm} (cf. \cite{cross})\label{duality}
Let $(\mathcal{D},\mathcal{C};\Gamma)$ be a cross-conection and let $(\mathcal{C},\mathcal{D};\Delta)$ be the dual cross-connection. Then there is a natural isomorphism $\chi_\Gamma$ between $\Gamma(-,-)$ and $\Delta(-,-)$ 
 such that $\chi_\Gamma : \Gamma(c,d) \to \Delta(c,d)$ is a bijection. 
\end{thm}
\begin{dfn}\label{ugamma}
Given a cross-connection $\Gamma$ of $\mathcal{D}$ with $\mathcal{C}$. Define
\begin{subequations}
\begin{align}
U\Gamma = & \bigcup\: \{ \quad \Gamma(c,d) : (c,d) \in v\mathcal{C} \times v\mathcal{D} \} \\
U\Delta = & \bigcup\: \{ \quad \Delta(c,d) : (c,d) \in v\mathcal{C} \times v\mathcal{D} \}
\end{align}
\end{subequations}
\end{dfn}
\begin{thm}(cf. \cite{cross}) \label{thmug}
For any cross-connection $\Gamma : \mathcal{D} \to N^\ast\mathcal{C}$, $U\Gamma$ is a regular subsemigroup of $T\mathcal{C}$ such that $\mathcal{C}$ is isomorphic to $\mathcal{L}(U\Gamma)$. And $U\Delta$ is a regular subsemigroup of $T\mathcal{D}$ such that $\mathcal{D}$ is isomorphic to $\mathcal{L}(U\Delta)$.
\end{thm}
\begin{dfn}
For a cross-connection $\Gamma : \mathcal{D} \to N^\ast\mathcal{C}$, we shall say that $\gamma \in U\Gamma$ is \emph{linked} to $\delta \in U\Delta$ if there is a $(c,d) \in v\mathcal{C} \times v\mathcal{D}$ such that $\gamma \in \Gamma(c,d)$ and $ \delta = \chi_\Gamma(c,d)(\gamma)$.
\end{dfn}
\begin{thm}(cf. \cite{cross})\label{thmcxs}
Let $\Gamma : \mathcal{D} \to N^\ast\mathcal{C}$ be a cross-connection and let 
$$ \tilde{S}\Gamma = \:\{\: (\gamma,\delta) \in U\Gamma\times U\Delta : (\gamma,\delta) \text{ is linked }\:\} $$ 
Then $\tilde{S}\Gamma$ is a regular semigroup with the binary operation defined by 
$$ (\gamma , \delta) \circ ( \gamma' , \delta') = (\gamma . \gamma' , \delta' . \delta)    $$
for all $(\gamma,\delta),( \gamma' , \delta') \in \tilde{S}\Gamma$. Then $\tilde{S}\Gamma$ is a sub-direct product of $U\Gamma$ and $U\Delta^{\text{op}}$ and is called the \emph{cross-connection semigroup} determined by $\Gamma$.
\end{thm}
Let $S$ be a regular semigroup. The category of principal left ideals of $S$ is described as follows.
Since every principal left ideal in
$S$ has at least one idempotent generator, we may write objects
(vertexes) in $\mathcal{L}(S)$ as $Se$ for $e\in E(S)$. Morphisms
$\rho:Se\to Sf$ are right translations $\rho=\rho(e,s,f)$ where $s \in eSf$ and $\rho$ maps $x \mapsto xs$. Thus
\begin{equation} \label{eqnLS}
  v\mathcal{L}(S) = \{ Se : e \in E(S)\}\quad\text{and}  \quad 
  \mathcal{L}(S) =\{\rho(e,s,f) : e,f \in E(S),\; s\in eSf\}.    
\end{equation}
\begin{pro} (cf. \cite{cross})\label{pro1} Let $S$ be a regular semigroup. Then $\mathcal{L}(S)$ is a normal category. $\rho(e,u,f)=\rho(e',v,f')$ if and only if $e \mathscr{L} e'$, $f\mathscr{L} f'$, $u \in eSf$, $v\in e'Sf'$ and $v=e'u$. Let $\rho=\rho(e,u,f)$ be a morphism in $\mathcal{L}(S)$. For any $g\in
    R_{u} \cap \omega(e)$ and $h\in E(L_{u})$, 
    $$ \rho=\rho(e,g,g)\rho(g,u,h)\rho(h,h,f) $$
     is a normal factorization of $\rho$.
\end{pro}
\begin{pro}(cf. \cite{cross} ) \label{pro2}
Let $S$ be a regular semigroup, $a \in S$ and $f \in E(L_a) $. Then for each $e \in E(S) $, let $\rho^a(Se) = \rho(e,ea,f) $. Then $\rho^a$ is a normal cone in $\mathcal{L}(S)$ with vertex $Sa$ called the principal cone generated by $a$. $M_{\rho^a} = \{ Se : e \in E(R_a) \}$. $\rho^a$ is an idempotent in $T\mathcal{L}(S)$ iff $a \in E(S)$.
The mapping $a \mapsto \rho^a$ is a homomorphism from S to $T\mathcal{L}(S)$. Further if S is a monoid, then S is isomorphic to $T\mathcal{L}(S)$.
\end{pro}
\begin{rmk} \label{rmkD}
  If $S^\text{{op}}$ denote the opposite semigroup of $S$ with multiplication
  given by $a\circ b = b.a$ where the right hand side is the product in
  $S$, then it is easy to see that $\mathcal{L}(S^\text{{op}})=\mathcal{R}(S)$ and
  $\mathcal{R}(S^\text{{op}})=\mathcal{L}(S)$. Using these it is possible to translate any
  statement about the category $\mathcal{R}(S)$ of right ideals of a semigroup
  $S$ as a statement regarding the category $\mathcal{L}(S^\text{{op}})$ of left
  ideals of the opposite semigroup $S^\text{{op}}$ and vice versa. Thus given a regular semigroup S, we have the following normal category of principal right ideals $\mathcal{R}(S)$ and the dual statements of the Proposition \ref{pro1} holds true for $\mathcal{R}(S)$.
\begin{equation}
  v\mathcal{R}(S) = \{R(e) = eS : e \in E(S)\}\quad\text{and}  \quad 
  \mathcal{R}(S) =\{\kappa(e,s,f) : e,f \in E(S),\; s\in fSe\}.    
\end{equation}
\end{rmk}
\section{Normal categories in a completely simple semigroup.}
Given a completely simple semigroup $S$, it is known that $S$ is isomorphic to a Rees matrix semigroup $S = \mathscr{M}[G;I,\Lambda;P]$ where $G$ is a group $I$ and $\Lambda$ are sets and $P = (p_{\lambda i})$ is a $\Lambda \times I$ matrix with entries in $G$(cf. \cite{rees}). Note that $S = G \times I \times \Lambda$ and the binary operation is given by 
$$(a,i,\lambda)(b,j,\mu) = (ap_{\lambda j}b,i,\mu) .$$
It is easy to see that $(g_1,i_1,\lambda_1) \mathscr{L} (g_2,i_2,\lambda_2)$ if and only if $\lambda_1 = \lambda_2$ (cf. \cite{howie}). Also $(a,i,\lambda)$ is an idempotent if and only if $a = (p_{\lambda i})^{-1}$. Observe that in this case, every principal left ideal\index{ideal!principal left} $Se$ where $e=(a,j,\lambda)$ is an idempotent can be described as $Se = \{ (g,i,\lambda)\: :\: g\in G,\:i \in I \}$. So
$Se$ depends only on $\lambda$ and we may write $ Se = G \times I \times \{\lambda\} $. Thus 
$$v\mathcal{L}(S) = \{ G \times I \times \{\lambda\} : \lambda \in \Lambda\}.$$ 
Henceforth we will denote the left ideal $Se = G \times I \times \lambda $ by $\bar{\lambda}$ and the set $v\mathcal{L}(S)$ will be denoted by $\bar{\Lambda}$.\\
Recall that any morphism from $Se=\bar{\lambda_1}$ to $Sf=\bar{\lambda_2}$ will be of the form $\rho(e,u,f)$ where $u \in eSf$ (see equation \ref{eqnLS}). Now for $e = (a,j,\lambda_1)$ and $f= (b,k,\lambda_2)$, it is easy to see that $eSf = \{ (k,j,\lambda_2) \: : \: k \in G \}$. Also for $u =(k,j,\lambda_2) \in eSf$ and $(h,i,\lambda_1)\in \bar{\lambda_1}$, we have $(h,i,\lambda_1)(k,j,\lambda_2) = (hp_{\lambda_1j}k,i,\lambda_2)$. Hence for $(h,i,\lambda_1) \in \bar{\lambda_1}$, the morphism $\rho(e,u,f)$ maps $(h,i,\lambda_1) \mapsto (hp_{\lambda_1j}k,i,\lambda_2)$ where $u= (k,j,\lambda_2)$. Observe that morphism involves only the right translation of the group element $h$ to $hp_{\lambda_1j}k$. And hence the morphism is essentially $h \mapsto hg$ such that $g \in G$ for some arbitrary $g\in G$ (taking $p_{\lambda_1j}k = g$). And so any morphism from $\bar{\lambda_1}$ to $\bar{\lambda_2}$ is determined by an element $g \in G$. We may denote this morphism by $\rho(g)$ so that 
$$\rho(g) : (h,i,\lambda_1)\mapsto(hg,i,\lambda_2).$$
Hence the set of morphisms from $\bar{\lambda_1}$ to $\bar{\lambda_2}$ can be represented by the set $G$. 
\begin{pro}\label{proiso}
Every morphism in $\mathcal{L}(S)$ is an isomorphism.
\end{pro}
\begin{proof}
By the discussion above any morphism $\rho(g):\bar{\lambda_1} \to \bar{\lambda_2}$ in $\mathcal{L}(S)$ is given by $ (h,i,\lambda_1)\mapsto(hp_{\lambda_1j}k,i,\lambda_2)$. Now we show that for $g'=p_{\lambda_2i}^{-1}k^{-1}p_{\lambda_1j}^{-1}$, $\rho(g'):\bar{\lambda_2} \to \bar{\lambda_1}$ is the right inverse of $\rho(g)$. To see that observe that $\rho(g)\rho(g'): (h,i,\lambda_1)\mapsto(hp_{\lambda_1j}k,i,\lambda_2)(p_{\lambda_2i}^{-1}k^{-1}p_{\lambda_1j}^{-1},i,\lambda_1) = (hp_{\lambda_1j}k p_{\lambda_2i} p_{\lambda_2i}^{-1}k^{-1}p_{\lambda_1j}^{-1},i,\lambda_1) =(h,i,\lambda_1)$. Also for $g''= k^{-1}p_{\lambda_1j}^{-1}p_{\lambda_1i}^{-1}$, we see that $\rho(g'')\rho(g):(h,i,\lambda_2)\mapsto (hk^{-1}p_{\lambda_1j}^{-1}p_{\lambda_1i}^{-1},i,\lambda_1)$ 
$(p_{\lambda_1j}k,i,\lambda_2) = (hk^{-1}p_{\lambda_1j}^{-1}p_{\lambda_1i}^{-1}p_{\lambda_1i}p_{\lambda_1j}k, i,\lambda_2) =(h,i,\lambda_2)$. And so $\rho(g''):\bar{\lambda_2} \to \bar{\lambda_1}$ is the left inverse of $\rho(g)$. Hence every morphism is an isomorphism in $\mathcal{L}(S)$.
\end{proof}
\begin{rmk}\label{rmkgrp}
Since every morphism in $\mathcal{L}(S)$ is an isomorphism\index{isomorphism}, $\mathcal{L}(S)$ is a \emph{groupoid}\index{groupoid}. So in $\mathcal{L}(S)$, all inclusions\index{inclusion} are identities and the epimorphic component\index{epimorphic component!of a morphism} of every morphism will be the morphism itself.   
\end{rmk}
Now we characterize the normal cones\index{cone!normal} in $\mathcal{L}(S)$. Recall that a cone in $\mathcal{C}$ with vertex $d \in v\mathcal{C}$ is a map $\gamma:v\mathcal{C}\to\mathcal{C}$ such that $\gamma(c)\in \mathcal{C}(c,d)$ for all $c\in v\mathcal{C}$ and whenever $c'\subseteq c$ then $j(c',c)\gamma(c) = \gamma(c')$. Since there are no non trivial inclusions in $\mathcal{L}(S)$, the second condition is redundant. So every cone in $\mathcal{L}(S)$ with vertex $\bar{\mu}$ is a mapping $\gamma : v\mathcal{L}(S) \to \mathcal{L}(S)$ such that for any $\bar{\lambda} \in \bar{\Lambda}$, $\gamma(\bar{\lambda}):\bar{\lambda} \to \bar{\mu}$ is a morphism in $\mathcal{L}(S)$. Let $\gamma(\bar{\lambda}) = \rho(g)$ for some $g =g(\lambda) \in G$. Now for any other $\bar{\lambda_2}\in \bar{\Lambda}$, we have $\gamma(\bar{\lambda_2}) = \rho(h)$ for some $h \in G$. Since no relation between $g$ and $h$ is required in describing cones in this case; any choice of $g(\lambda) \in G$ for $\lambda \in \Lambda$ produces a normal cone. Thus a cone with vertex $\bar{\lambda}$ can be represented as $(\bar{\gamma},\lambda)$ where $\bar{\gamma} \in G^{\Lambda}$. Also since every morphism is an isomorphism, every cone in $\mathcal{L}(S)$ will be a normal cone.\\
Now $G^\Lambda$ has the structure of a group as the direct product of $G$ over $\Lambda$. And $\Lambda$ (coming from the Rees matrix semigroup) is a right zero semigroup.
\begin{lem}\label{lemsgls}
$G^\Lambda \times \Lambda$ is a semigroup with the binary operation defined as follows. Given $\gamma_1 = (\bar{\gamma_1},\lambda_k)$, $\gamma_2 = (\bar{\gamma_2},\lambda_l) \in G^\Lambda \times \Lambda$
\begin{equation}\label{eqnsgls}
\gamma_1 \ast \gamma_2  = (\bar{\gamma_1}.\bar{g_k} ,\lambda_l)
\end{equation}
where $g_k = \bar{\gamma_2}(\lambda_k)$ and $\bar{g_k} = (g_k,g_k,g_k,...) \in G^\Lambda$.
\end{lem}
\begin{proof}
Observe $(\bar{\gamma_1}.\bar{g_k} ,\lambda_l) \in G^\Lambda \times \Lambda$ and so clearly $\ast$ is a well-defined binary operation on $G^\Lambda \times \Lambda$.\\
Now if $\gamma_1 = (\bar{\gamma_1},\lambda_k)$, $\gamma_2 = (\bar{\gamma_2},\lambda_l)$, $\gamma_3 =(\bar{\gamma_3},\lambda_m) \in G^\Lambda \times \Lambda$; then 
$$(\gamma_1 \ast \gamma_2)\ast \gamma_3  = (\bar{\gamma_1}.\bar{g_k} ,\lambda_l)\ast (\bar{\gamma_3},\lambda_m) = ((\bar{\gamma_1}.\bar{g_k}).\bar{h_l} ,\lambda_m)$$
where $h_l = \bar{\gamma_3}(\lambda_l)$. And as $\bar{\gamma_2}\bar{h_l}(\lambda_k) = \bar{g_k}.\bar{h_l}$,
$$\gamma_1 \ast (\gamma_2\ast \gamma_3)  = (\bar{\gamma_1},\lambda_k)\ast (\bar{\gamma_2}\bar{h_l},\lambda_m) = (\bar{\gamma_1}.(\bar{g_k}.\bar{h_l}) ,\lambda_m)$$
Now since $(\bar{\gamma_1}.\bar{g_k}).\bar{h_l} = \bar{\gamma_1}.(\bar{g_k}.\bar{h_l}) = \bar{\gamma_1}.\bar{g_k}.\bar{h_l}$ (being multiplication in the group $G^\Lambda$ and hence associative); we have 
$$(\gamma_1 \ast \gamma_2)\ast \gamma_3 = \gamma_1 \ast (\gamma_2\ast \gamma_3) \text{ for all } \gamma_1,\gamma_2,\gamma_3 \in G^\Lambda \times \Lambda.$$ 
Hence $\ast$ is associative and so $(G^\Lambda \times \Lambda, \ast)$ is a semigroup.
\end{proof}
\begin{pro}\label{prosgls}
The semigroup\index{semigroup!of normal cones} $T\mathcal{L}(S)$ of all normal cones in $\mathcal{L}(S)$ is isomorphic to $G^\Lambda \times \Lambda$.
\end{pro}
\begin{proof}
Since a normal cone with vertex $\bar{\lambda}$ can be represented as $(\bar{\gamma},\lambda)$ where $\bar{\gamma} \in G^{\Lambda}$, $T\mathcal{L}(S) \subseteq G^\Lambda \times \Lambda$. Conversely since any $(\bar{\gamma},\lambda) \in G^\Lambda \times \Lambda$ represents a cone in $\mathcal{L}(S)$, $G^\Lambda \times \Lambda \subseteq T\mathcal{L}(S)$ and hence $G^\Lambda \times \Lambda = T\mathcal{L}(S)$.\\ 
Now the multiplication in $T\mathcal{L}(S)$ is defined as for any $\lambda \in \Lambda$, $(\gamma_1 * \gamma_2)(\lambda) = \gamma_1(\lambda) (\gamma_2(c_\gamma))^\circ$ (see equation \ref{eqnsg}) for $\gamma_1, \gamma_2 \in T\mathcal{L}(S)$. Let $\gamma_1 = (\bar{\gamma_1},\lambda_k)$, $\gamma_2 =(\bar{\gamma_2},\lambda_l) \in G^\Lambda \times \Lambda$,
$\gamma_2(c_{\gamma_1}) = (\bar{g_k},\lambda_l)$ where $g_k = \bar{\gamma_2}(\lambda_k)$ and $\bar{g_k} = (g_k,g_k,g_k,...) \in G^\Lambda$. And since the epimorphic component of any morphism in $\mathcal{L}(S)$ is itself $\gamma_2(c_{\gamma_1})^\circ = (\bar{g_k},\lambda_l)$.\\
Therefore
$$ \gamma_1 \ast \gamma_2  = \gamma_1 .(\gamma_2(c_{\gamma_1}))^\circ =  (\bar{\gamma_1},\lambda_k).(\bar{g_k},\lambda_l) = (\bar{\gamma_1}.\bar{g_k} ,\lambda_l). $$
And hence $T\mathcal{L}(S)$ is isomorphic to $G^\Lambda \times \Lambda$.
\end{proof}
We further observe that the semigroup obtained here can be realised as a semi-direct product of semigroups.
\begin{dfn} \label{lac}
Let S and T be semigroups. A \emph{(left) action} of T on S is a map $\psi: T \times S \to S$,
satisfying: (i) $({t_1t_2},s)\psi =\: (t_1,(t_2,s)\psi)\psi$ and\\
(ii) $({t},s_1s_2)\psi =\: ({t},s_1)\psi(t,s_2)\psi$ for all $t,t_1,t_2 \in T$ and $s,s_1,s_2 \in S$. We denote $(t,s)\psi = t \ast s$.
\end{dfn}
\begin{dfn} \label{sdp}
Let S and T be semigroups. The \emph{semidirect product}\index{semigroup!semidirect product} $S \ltimes T$ of S and T , with respect to a left action $\psi$ of $T$ on $S$ is defined as $S \times T$ with multiplication given by 
$$(s_1,t_1)(s_2,t_2)\: =\: (s_1\:(t_1 \ast s_2),t_1t_2)$$
\end{dfn}
It is well known that $S\ltimes T$ is a semigroup. It is trivially verified that the idempotents
in $S\ltimes T$ are the pairs $(s,t)$ such that $t \in E(T)$ and $s(t \ast s) = s$.\\
Now we show that the semigroup of normal cones in $\mathcal{L}(S)$ is isomorphic to a semi-direct product $G^\Lambda \ltimes \Lambda$ of $G^\Lambda $ and $\Lambda$.\\
Firstly, we look at the semigroups $G^\Lambda$ and $\Lambda$. Since $G$ is a group; $G^\Lambda$ will form a group under component-wise multiplication defined as follows.\\
For $(g_1,g_2...)$,$(h_1,h_2...) \in G^\Lambda$ 
$$(g_1,g_2...)(h_1,h_2...) \: = \: (g_1h_1,g_2h_2,...) $$
and hence in particular $G^\Lambda $ is also a semigroup. The set $\Lambda$ admits a right zero semigroup structure and hence has an in-built multiplication given by $\lambda_k \lambda_l = \lambda_l $ for every $\lambda_k, \lambda_l \in \Lambda$.\\ 
Now we define a left action of $\Lambda$ on $G^\Lambda$ as follows. For $\lambda_k \in \Lambda$ and $\bar{g} = (g_1,g_2,...) \in G^\Lambda$,
\begin{equation}\label{eqnlac}
\lambda_k \ast \bar{g} \:=\: (g_k,g_k,...)
\end{equation}
where $g_k = \bar{g}(\lambda_k)$.
\begin{lem}\label{lemlac}
The function as defined in equation \ref{eqnlac} is a left action of $\Lambda$ on $G^\Lambda$. 
\end{lem}
\begin{proof}
Clearly the function is well-defined. Now for $\bar{g} = (g_1,g_2,...) \in G^\Lambda$ and $\lambda_k, \lambda_l \in \Lambda $, since $\lambda_k\lambda_l = \lambda_l$, 
$$\lambda_k\lambda_l\ast\bar{g} = \lambda_k\lambda_l\ast(g_1,g_2,...) = \lambda_l\ast(g_1,g_2,...) = (g_l,g_l,...).$$
Also $$\lambda_k\ast(\lambda_l\ast(\bar{g}))\:=\:\lambda_k\ast(\lambda_l\ast(g_1,g_2,...)) \: = \: \lambda_k\ast(g_l,g_l,...) \: = \:  (g_l,g_l,...).$$
Hence $\lambda_k\lambda_l\ast\bar{g}\: =\: \lambda_k\ast(\lambda_l\ast(\bar{g})) $ for $\bar{g} \in G^\Lambda$ and $\lambda_k,\lambda_l \in \Lambda$.\\
Then for $\bar{g} = (g_1,g_2,...), \bar{h} = (h_1,h_2,...) \in G^\Lambda$, and $\lambda_k \in \Lambda$,
$$\lambda_k \ast(\bar{g}\bar{h}) \:=\: \lambda_k \ast ((g_1,g_2...)(h_1,h_2...),)) \:=\: \lambda_k \ast(g_1h_1,g_2h_2,...)\:=\: (g_kh_k,g_kh_k,...).$$
And\\
$(\lambda_k \ast \bar{g}) (\lambda_k\ast \bar{h}) \:=\: (\lambda_k \ast (g_1,g_2,...))(\lambda_k \ast (h_1,h_2,...)) \:$
\flushright $ = \: (g_k,g_k,...)(h_k,h_k,...) \: = \: (g_kh_k,g_kh_k,...)$.\\
\flushleft And hence $\lambda_k \ast(\bar{g}\bar{h}) \:=\: (\lambda_k \ast \bar{g}) (\lambda_k\ast \bar{h})$ for $\bar{g} , \bar{h} \in G^\Lambda$, and $\lambda_k \in \Lambda$.\\
Thus equation \ref{eqnlac} is a left semigroup action of $\Lambda$ on $G^\Lambda$.
\end{proof}
\begin{pro}
$T\mathcal{L}(S)$ is the semi-direct\index{semigroup!semidirect product} product $G^\Lambda \ltimes \Lambda$ with respect to the left action above.
\end{pro}
\begin{proof}
The semi-direct product of $G^\Lambda \ltimes \Lambda$ with respect to the left action is given as follows. For $(\bar{\gamma_1},\lambda_k)$, $(\bar{\gamma_2},\lambda_l) \in G^\Lambda \times \Lambda$,
$$(\bar{\gamma_1},\lambda_k) \ast (\bar{\gamma_2},\lambda_l) \:=\: (\bar{\gamma_1}.(\lambda_k \ast\bar{\gamma_2}) , \lambda_k\lambda_l) \:=\: (\bar{\gamma_1}.\bar{g_k} ,\lambda_l)$$
where $g_k = \bar{\gamma_2}(\lambda_k)$ and $\bar{g_k} = (g_k,g_k,...)$. So if $\gamma_1 = (\bar{\gamma_1},\lambda_k)$ and $\gamma_2 = (\bar{\gamma_2},\lambda_l)$ then the multiplication defined above is exactly the same multiplication defined in equation \ref{eqnsgls} and hence $T\mathcal{L}(S)$ is the semi-direct product $G^\Lambda \ltimes \Lambda$ with respect to the left action above. 
\end{proof}
Now we proceed to characterize the principal cones in $T\mathcal{L}(S)$. 
\begin{pro}\label{proprinci}
The principal cones in $T\mathcal{L}(S)$ forms a regular subsemigroup of $G^\Lambda \ltimes \Lambda$.
\end{pro}
\begin{proof}
Given $a=(g_a,i_a,\lambda_a)$, the principal cone $\rho^a$ (see Proposition \ref{pro2}) will be a normal cone with vertex $\bar{\lambda_a}$ such that each left ideal $\bar{\lambda_k}$ is right multiplied by $(g_a,i_a,\lambda_a)$. But since the morphisms in $\mathcal{L}(S)$ involves right multiplication by the product of a sandwich element of the matrix and the group element; at each $\bar{\lambda_k} \in \Lambda$, $\rho^a(\bar{\lambda_k})\:=\: p_{\lambda_k i_a}g_a$ (where the sandwich matrix $P = (p_{\lambda i})_{\lambda \in \Lambda , i\in I}$ ).\\
Hence $\rho^a $ can be represented by $(p_{\lambda_1 i_a}g_a,p_{\lambda_2 i_a}g_a,p_{\lambda_3 i_a}g_a, ... ; \lambda_a) \in G^\Lambda \ltimes \Lambda$. Observe that $\bar{\rho^a} \in G^\Lambda$ is the right translation of the $i_a$-th column of the sandwich matrix $P$ with group element $g_a$.\\
Now $\rho^a . \rho^b = (p_{\lambda_1 i_a}g_a,p_{\lambda_2 i_a}g_a,p_{\lambda_3 i_a}g_a, ... , \lambda_a).(p_{\lambda_1 i_b}g_b,p_{\lambda_2 i_b}g_b,p_{\lambda_3 i_b}g_b, ... , \lambda_b) \:$
\flushright$= \:(p_{\lambda_1 i_a}g_ap_{\lambda_a i_b}g_b,p_{\lambda_2 i_a}g_ap_{\lambda_a i_b}g_b,p_{\lambda_3 i_a}g_ap_{\lambda_a i_b}g_b, ... ,\lambda_b)$.
\flushleft Also $a.b = (g_a,i_a,\lambda_a)(g_b,i_b,\lambda_b) = (g_ap_{\lambda_ai_b}g_b,i_a,\lambda_b)$.\\
So $\rho^{ab} = (p_{\lambda_1 i_a}g_ap_{\lambda_ai_b}g_b,p_{\lambda_2 i_a}g_ap_{\lambda_ai_b}g_b,p_{\lambda_3 i_a}g_ap_{\lambda_ai_b}g_b, ... , \lambda_b)$.\\
Hence $\rho^a . \rho^b = \rho^{ab}$ and consequently the map $a\mapsto\rho^a$ from $S\to T\mathcal{L}(S)$ is a homomorphism. So the set of principal cones in $\mathcal{L}(S)$ forms a subsemigroup of $T\mathcal{L}(S)$. And since $S$ is regular , the semigroup of principal cones forms a regular subsemigroup of $G^\Lambda \ltimes \Lambda$.
\end{proof}
If $S$ is a regular monoid, then it is known that $S$ is isomorphic to $T\mathcal{L}(S)$(see Proposition \ref{pro2}). In general, if $S$ is not a monoid this may not be true. But always there exists a homomorphism from $S \to T\mathcal{L}(S)$ mapping $a \mapsto \rho^a$ which may not be injective or onto. The map is not injective when $S$ is a rectangular band. And the above discussion tells us that if $S$ is a completely simple semigroup, the map is never onto $T\mathcal{L}(S)$. But in the case of completely simple semigroups, there exists some special cases where this map is infact injective and consequently $S$ can be realized as a subsemigroup of $T\mathcal{L}(S)$. 
\begin{thm}\label{thmiso}
$S = \mathscr{M}[G;I,\Lambda;P]$ is isomorphic to the semigroup of principal cones in $T\mathcal{L}(S)$ if and only if for every $g \in G$ and $i_1\neq i_2 \in I$, there exists $\lambda_k \in \Lambda$ such that $p_{\lambda_k i_1} \neq p_{\lambda_k i_2}g$.
\end{thm}
\begin{proof}
As seen above, there exists a semigroup homomorphism from $\psi: S \to T\mathcal{L}(S)$ mapping $a \mapsto \rho^a$ . Now $S$ can be seen as a subsemigroup of $T\mathcal{L}(S)$ if and only if $\psi$ is injective.\\
We claim $\psi$ is injective only if for every $i_1\neq i_2 \in I$ and every $g \in G$, there exists $\lambda_k \in \Lambda$ such that $p_{\lambda_k i_1} \neq p_{\lambda_k i_2}g$. Suppose not. .ie there exists $i_1\neq i_2 \in I$ and a $g \in G$ such that $p_{\lambda_k i_1} = p_{\lambda_k i_2}g$ for every $\lambda_k \in \Lambda$. Then without loss of generality, assume $g=g_2g_1^{-1}$ for some $g_1,g_2 \in G$. Then $p_{\lambda_k i_1} = p_{\lambda_k i_2}g_2g_1^{-1}$ for some $g_1,g_2 \in G$. .ie $p_{\lambda_k i_1}g_1 = p_{\lambda_k i_2}g_2$ for every $\lambda_k \in \Lambda$. And hence if $a =(g_1,i_1,\lambda)$ and $b =(g_2,i_2,\lambda)$; then since $i_1 \neq i_2$, $a \neq b$. Then $\rho^a = (p_{\lambda_1 i_1}g_1,p_{\lambda_2 i_1}g_1,p_{\lambda_3 i_1}g_1, ... , \lambda)$ and $\rho^b = (p_{\lambda_1 i_2}g_2,p_{\lambda_2 i_2}g_2,p_{\lambda_3 i_2}g_2, ... , \lambda)$. And since $p_{\lambda_k i_1}g_1 = p_{\lambda_k i_2}g_2$ for every $\lambda_k \in \Lambda$, $a \neq b$ but $\rho^a = \rho^b$. And hence $\psi$ is not injective. And hence $S$ is not isomorphic to the semigroup of principal cones in $T\mathcal{L}(S)$.\\
Conversely suppose for every $i_1\neq i_2 \in I$ and every $g \in G$, there exists $\lambda_k \in \Lambda$ such that $p_{\lambda_k i_1} \neq p_{\lambda_k i_2}g$, we need to show $\psi $ is injective so that $S$ is isomorphic to the semigroup of principal cones. Suppose if $a =(g_a,i_a,\lambda_a)$ and $b = (g_b,i_b,\lambda_b)$ and $\psi(a) =\psi(b)$. So $\rho^a =\rho^b$ ; .ie $(p_{\lambda_1 i_a}g_a,p_{\lambda_2 i_a}g_a,p_{\lambda_3 i_a}g_a, ... ,\lambda_a) = (p_{\lambda_1 i_b}g_b,p_{\lambda_2 i_b}g_b,p_{\lambda_3 i_b}g_b, ... , \lambda_b)$. So clearly $\lambda_a =\lambda_b$. And $p_{\lambda_k i_a}g_a = p_{\lambda_k i_b}g_b$ for every $ \lambda_k \in \Lambda$. Now if $g_a \neq g_b$, then $p_{\lambda_k i_a} = p_{\lambda_k i_b}g_bg_a^{-1}$ and so $p_{\lambda_k i_a} = p_{\lambda_k i_b}g$ for every $ \lambda_k \in \Lambda$(taking $g =g_bg_a^{-1}$). But this will contradict our supposition unless $i_a =i_b$. But then $p_{\lambda_k i_a} = p_{\lambda_k i_a}g$ for every $ \lambda_k \in \Lambda$. Now this is possible only if $g=e$ ; which implies $g_bg_a^{-1}= e$ .ie $g_b= g_a $; which is again a contradiction. Hence $g_a =g_b$.\\
Now if $i_a \neq i_b$; then we have  $p_{\lambda_k i_a}g_a = p_{\lambda_k i_b}g_b$ for every $ \lambda_k \in \Lambda$. Then taking $g = g_bg_a^{-1}$; we have  $p_{\lambda_k i_a} = p_{\lambda_k i_b}g$ for every $ \lambda_k \in \Lambda$ ; which is again a contradiction to our supposition. Hence $i_a = i_b$.\\
So $(g_a,i_a,\lambda_a) = (g_b,i_b,\lambda_b)$ and so $a=b$. Thus $\rho^a = \rho^b$ implies $a=b$ making $\psi$ injective. So if for every $i_1\neq i_2 \in I$ and every $g \in G$, there exists $\lambda_k \in \Lambda$ such that $p_{\lambda_k i_1} \neq p_{\lambda_k i_2}g$, then $S$ is isomorphic to the semigroup of principal cones. Hence the proof.
\end{proof}
Now we proceed to characterize the Green's relations\index{Green's relation} in $T\mathcal{L}(S)$.
\begin{pro}\label{proltls}
If $\gamma_1=(\bar{\gamma_1},\lambda_k),\gamma_2=(\bar{\gamma_2},\lambda_l) \in T\mathcal{L}(S)$, then $\gamma_1\mathscr{L}\gamma_2$ if and only if $\lambda_k =\lambda_l$.
\end{pro}
\begin{proof}
Suppose if $\lambda_k \neq \lambda_l$. Then for an arbitrary $\gamma = (\bar{\gamma},\lambda_m) \in T\mathcal{L}(S)$, $$\gamma\gamma_1 = (\bar{\gamma},\lambda_m)(\bar{\gamma_1},\lambda_k) = (\bar{\gamma}.\bar{g_m},\lambda_k)\text{ where }g_m = \gamma_1(\bar{\lambda_m}).$$ Hence $\gamma\gamma_1 \in G^\Lambda \times \lambda_k$. And since $\bar{\gamma} \in G^\Lambda$ is arbitrary, $T\mathcal{L}(S)\gamma_1 = G^\Lambda \times \lambda_k $.
Similarly $T\mathcal{L}(S)\gamma_2 = G^\Lambda \times \lambda_l $. And so $T\mathcal{L}(S)\gamma_1 \neq T\mathcal{L}(S)\gamma_2 $.\\
Conversely if $\lambda_k =\lambda_l$, then $G^\Lambda \times \lambda_k = G^\Lambda \times \lambda_l$ and so $T\mathcal{L}(S)\gamma_1 = T\mathcal{L}(S)\gamma_2$. Thus $\gamma_1\mathscr{L}\gamma_2$. Hence the proof. 
\end{proof}
Now we proceed to get a characterization of Green's $\mathscr{R}$ relation in the semigroup $T\mathcal{L}(S)$. This will also give the characterization for the normal dual $N^\ast\mathcal{L}(S)$. For this end, begin by observing that $G^\Lambda$ is a group with component-wise multiplication defined as follows. For $(g_1,g_2...)$,$(h_1,h_2...) \in G^\Lambda$ 
$$(g_1,g_2...)(h_1,h_2...) \: = \: (g_1h_1,g_2h_2,...) $$
Then $G$ may be viewed as a subgroup (not necessarily normal) of $G^\Lambda$ by identifying $g \mapsto
(g,g,...) \in G^\Lambda$.  Then for some $\bar{\gamma} = (g_1,g_2...) \in G^\Lambda$ we look at the left coset $\bar{\gamma} G$ of $G$ in $G^\Lambda$ with respect to $\bar{\gamma}$. It is defined as 
$$ \bar{\gamma} G  = \{(g_1g,g_2g,...) \: : \: g \in G\}$$
Observe that these cosets form a partition of $G^\Lambda$. Now we show that these left cosets of $G$ in $G^\Lambda$ infact gives the characterization of Green's $\mathscr{R}$ relation in $T\mathcal{L}(S)$.\\
Let $\gamma = (\bar{\gamma},\lambda_m) \in T\mathcal{L}(S)$; then $\gamma_1\gamma = (\bar{\gamma_1},\lambda_k)(\bar{\gamma},\lambda_m) = (\bar{\gamma_1}.\bar{g_k},\lambda_m)$ where $g_k = \gamma(\bar{\lambda_k})$. Hence $\gamma_1\gamma \in ( \bar{\gamma_1}G )\times \Lambda$. And since ${\gamma} \in G^\Lambda \times \Lambda$ is arbitrary, both $\lambda_m$ and $g_k$ can be arbitrarily chosen.  And hence $\gamma_1T\mathcal{L}(S) = ( \bar{\gamma_1}  G ) \times \Lambda $. And this gives us the following characterization of the $\mathscr{R}$ relation in $T\mathcal{L}(S)$. 
\begin{pro}\label{prortls}
For $\gamma_1 =(\bar{\gamma_1},\lambda_k), \gamma_2= (\bar{\gamma_2},\lambda_l) \in T\mathcal{L}(S)$, $\gamma_1 \mathscr{R} \gamma_2$ if and only if $\bar{\gamma_1}  G = \bar{\gamma_2}  G$.
\end{pro}
\begin{proof}
Suppose $\bar{\gamma_1}  G = \bar{\gamma_2}  G$; then $ ( \bar{\gamma_1}  G ) \times \Lambda  =  ( \bar{\gamma_2}  G ) \times \Lambda$ and by the above discussion $\gamma_1T\mathcal{L}(S) = \gamma_2T\mathcal{L}(S)$ and hence $\gamma_1 \mathscr{R} \gamma_2$.\\
Conversely suppose $\bar{\gamma_1}  G \neq \bar{\gamma_2}  G$. i.e there exists a $\gamma = (\bar{\gamma},\lambda_m) \in T\mathcal{L}(S)$ such that $\bar{\gamma} \in \bar{\gamma_1}  G$ but $\bar{\gamma} \notin \bar{\gamma_2}  G$. And hence $\gamma \in ( \bar{\gamma_1}  G ) \times \Lambda $ but $\gamma \notin ( \bar{\gamma_2}  G ) \times \Lambda $. And hence $\gamma_1T\mathcal{L}(S) \neq \gamma_2T\mathcal{L}(S)$. And so $\gamma_1$ and $\gamma_2$ are not $\mathscr{R}$ related. \\
Hence the proof.
\end{proof}
\begin{rmk}
A very important remark is to be made. Recall from Theorem \ref{thm1} that for $\gamma_1,\gamma_2 \in T\mathcal{C}$ and if $\gamma_1 \mathscr{R} \gamma_2$, then $M{\gamma_1} = M{\gamma_2}$. The discussion above provides a counter example for the converse. Observe that if we take $\mathcal{C} = \mathcal{L}(S)$ where $S$ is a completely simple semigroup, then by Remark \ref{rmkgrp}, $\mathcal{L}(S)$ is a groupoid. And hence $M\gamma = v\mathcal{L}(S)$ for every $\gamma \in T\mathcal{L}(S)$. But by Proposition \ref{prortls}, for $\gamma_1 =(\bar{\gamma_1},\lambda_k), \gamma_2= (\bar{\gamma_2},\lambda_l) \in T\mathcal{L}(S)$, $\gamma_1 \mathscr{R} \gamma_2$ if and only if $\bar{\gamma_1}  G = \bar{\gamma_2}  G$. So if we take $\gamma_1 , \gamma_2 \in T\mathcal{L}(S)$ such that $\bar{\gamma_1}  G \neq \bar{\gamma_2}  G$, then $M{\gamma_1} = M{\gamma_2} = \bar{\Lambda}$ but $\gamma_1 $ is not $\mathscr{R}$ related to $ \gamma_2$.
\end{rmk}
Now we proceed to describe the normal dual\index{normal dual} $N^\ast\mathcal{L}(S)$.
\begin{pro}\label{pro nd}
The normal dual $N^\ast\mathcal{L}(S)$ is a category such that $vN^\ast\mathcal{L}(S)$ is the left cosets of $G$ in $G^\Lambda$ and there is a bijection $\psi$ from the set of morphisms in $N^\ast\mathcal{L}(S)$ to the group $G$ such that $(\sigma)\psi \circ (\tau)\psi = (\tau \circ\sigma) \psi$ where $\sigma$ and $\tau$ are morphisms in $N^\ast\mathcal{L}(S)$.
\end{pro}
\begin{proof} 
By Proposition \ref{prortls}, $v\mathcal{R}(T\mathcal{L}(S))$ is the set of all left cosets of $G$ in $G^\Lambda$ and Theorem \ref{thm1}, $\mathcal{R}(T\mathcal{L}(S))$ is isomorphic to $N^\ast\mathcal{L}(S)$ as normal categories. Hence the normal dual $N^\ast\mathcal{L}(S)$ of $\mathcal{L}(S)$ is the category whose objects are the left cosets of $G$ in $G^\Lambda$ with morphisms $\sigma_h$ described as follows.\\ 
Recall that by Theorem \ref{thm1}, to every morphism $\sigma : H(\gamma_1;-) \to H(\gamma_2;-)$ in $N^\ast\mathcal{L}(S)$, there is a unique $\hat{\sigma} : c_{\gamma_2} \to c_{\gamma_1}$ in $\mathcal{L}(S)$ such that the component of the natural transformation $\sigma$ at $Se \in v\mathcal{L}(S) $ is the map given by $\sigma(Se) : \gamma_1 \ast f^\circ \mapsto \gamma_2 \ast (\hat{\sigma} f)^\circ $.
Hence if $\gamma_1 =(\bar{\gamma_1},\lambda_k)$ and $\gamma_2= (\bar{\gamma_2},\lambda_l)$, corresponding to the morphism $\sigma : \bar{\gamma_1} G \to \bar{\gamma_2} G$, there exists a unique morphism $\rho(h):\bar{\lambda_l} \to \bar{\lambda_k}$, such that for $\bar{\gamma_1} g \in \bar{\gamma_1} G $ and for some $\rho(g) :\bar{\lambda_k} \to \bar{\lambda}$
$$\sigma(\bar{\lambda}) : \bar{\gamma_1} g \mapsto \bar{\gamma_2} h g.$$ 
We will denote this morphism $\sigma : \bar{\gamma_1} G \to \bar{\gamma_2} G$ by $\sigma_h$. Also given $h \in G$, we get a morphism $\sigma_h:\bar{\gamma_1} G \to \bar{\gamma_2} G$ as $\sigma_h :\bar{\gamma_1} g \mapsto \bar{\gamma_2} h g$. And hence the map $ \psi: \sigma_h\mapsto h$ is a bijection from $N^\ast\mathcal{L}(S)$ to $G$.\\ 
Now if $\gamma_1 =(\bar{\gamma_1},\lambda_k)$, $\gamma_2= (\bar{\gamma_2},\lambda_l)$ and $\gamma_3= (\bar{\gamma_3},\lambda_m)$, for $\sigma_h : \bar{\gamma_1} G \to \bar{\gamma_2} G$ and $\sigma_k : \bar{\gamma_2} G \to \bar{\gamma_3} G$, $(\bar{\gamma_1} g)\sigma_h \circ \sigma_k = (\bar{\gamma_2}  h g)\sigma_k = \bar{\gamma_3} k h g$.\\
Also $(\bar{\gamma_1} g) \sigma_{k h} = \bar{\gamma_3} k h g$ (by the uniqueness of $\rho(kh):\bar{\lambda_m} \to \bar{\lambda_k}$).\\
So $\sigma_h \circ \sigma_k = \sigma_{k h}$ and thus $(\sigma_h \circ \sigma_k)\psi = (\sigma_{k h})\psi = k h = (\sigma_k)\psi \circ (\sigma_h)\psi$.\\
Hence the result.
\end{proof}
Observe that the $\mathscr{R}$ relation of $\gamma_1 =(\bar{\gamma_1},\lambda_k) \in T\mathcal{L}(S)$ depends only on $\bar{\gamma_1}$ and not on $\lambda_k$ and so we will denote the $\mathscr{R}$ classes of $T\mathcal{L}(S)$ by $R_{\bar{\gamma}}$ such that $\bar{\gamma} \in G^\Lambda$. 
\begin{rmk} \label{rmkrs}
By Remark \ref{rmkD},\index{dual} we can use the above discussion to characterize the normal categories $\mathcal{R}(S)$ and $N^\ast\mathcal{R}(S)$. Any principal right ideal of $S$ will be of the form $G\times \{i\} \times \Lambda$ such that $i \in I$. And so
$$v\mathcal{R}(S) = \{ G \times \{i\} \times \Lambda : i \in I\}.$$ 
Henceforth we will denote the right ideal $(a,i,\lambda)S = G \times \{i\} \times \Lambda $ by $\bar{i}$ and the set $v\mathcal{R}(S)$ will be denoted by $\bar{I}$.
Any morphism in $\mathcal{R}(S)$ from $eS=\bar{i_1}$ to $fS=\bar{i_2}$ is of the form $\kappa(e,u,f)$ where $u \in fSe$. As argued for $\mathcal{L}(S)$, any morphism from $\bar{i_1}$ to $\bar{i_2}$ maps $(h,i_1,\lambda) \mapsto (gh,i_2,\lambda)$ for $g \in G$
and the set of morphisms from $\bar{i_1}$ to $\bar{i_2}$ is
$$\mathcal{R}(S)(\bar{i_1},\bar{i_2}) = \{ (h,i_1,\lambda)\mapsto(gh,i_2,\lambda) : g \in G \}.$$
But since the translation is on the left, the set of morphisms from $\bar{i_1}$ to $\bar{i_2}$ is bijective with the group $G^{\text{op}}$ and each $g \in G$ will map $(h,i_1,\lambda) \mapsto (gh,i_2,\lambda)$. Henceforth we will denote this morphism as $\kappa(g)$ when there is no ambiguity regarding the domain and range.\\
Proceeding the same way as for $\mathcal{L}(S)$, we can show that $T\mathcal{R}(S)$ the semigroup of normal cones in $\mathcal{R}(S)$ is isomorphic to $(G^I \ltimes I)^{\text{op}}$ the opposite of semi-direct product of $G^I$ and $I$ with respect to the left action defined as follows.
$$i_k\ast(g_1,g_2,...)\:=\: (g_k,g_k,...)$$
And hence given $\delta_1 = (\bar{\delta_1},i_k)$, $\delta_2 = (\bar{\delta_2},i_l) \in G^I \times I = T\mathcal{R}(S)$, the multiplication is defined as
\begin{equation}\label{eqnsgrs}
\delta_1 \ast \delta_2  = (\bar{g_k}.\bar{\delta_1},i_l)
\end{equation}
where $g_k = \bar{\delta_2}(\lambda_k)$ and $\bar{g_k} = (g_k,g_k,...) \in G^I$. \\
We also see that the normal dual $N^\ast\mathcal{R}(S)$ is isomorphic to $G^I/G$, the $\bf{right}$ cosets of $G$ in $G^I$ and a morphism denoted by $\tau_h : G\bar{\delta_1}\to G\bar{\delta_2}$ maps $g\bar{\delta_1} \mapsto gh\bar{\delta_2}$. And consequently the set of morphisms is bijective with the group $G$.
\end{rmk}

\section{Cross-connections of completely simple semigroups}
We have characterized the normal categories associated with a completely simple semigroup. Now we proceed to construct completely simple semigroup as the cross-connection semigroup using the appropriate \emph{local isomorphisms} between the categories.\\
Firstly, recall that any cross-connection $\Gamma$ from $\mathcal{R}(S) \to N^\ast\mathcal{L}(S)$ is a local isomorphism such that for every $c \in v\mathcal{L}(S)$, there is some $d \in v\mathcal{R}(S)$ such that $c \in M\Gamma(d)$ .ie $\Gamma$ is a local isomorphism from $\bar{I} \to G^\Lambda/G$ such that for every $\bar{\lambda} \in \bar{\Lambda}$, there is some $\bar{i} \in \bar{I}$ such that $\bar{\lambda} \in M\Gamma(\bar{i})$.\\
We define $\Gamma : \bar{I} \to G^\Lambda/G$ such that
\begin{equation}\label{eqngamma}
v\Gamma:\bar{i}  \mapsto a_i  G \qquad \Gamma:\: \kappa(g) \mapsto \sigma_g 
\end{equation} 
for any $a_i \in G^\Lambda$.
\begin{pro}\label{prolociso}
The functor $\Gamma$ defined in equation \ref{eqngamma} is a local isomorphism\index{local isomorphism}.
\end{pro}
\begin{proof}
Firstly if $\bar{i_1} = \bar{i_2}$, then $i_1 =i_2$ and hence $v\Gamma$ is well-defined. Also if $\kappa({g_1}) = \kappa({g_2})$; then ${g_1} = g_2$ and hence $\Gamma$ is well-defined on morphisms as well. Also $\Gamma(\kappa(g_1).\kappa(g_2))  = \Gamma(\kappa(g_2g_1))  = \sigma_{g_2g_1}= \sigma_{g_1}.\sigma_{g_2}$. Hence $\Gamma$ is a well-defined covariant functor.\\
Since there are no non trivial inclusions in either categories, $\Gamma$ is trivially inclusion preserving. Since the set of morphisms between any two objects is bijective with $G^{\text{op}}$ in both cases, $\Gamma$ is fully-faithful. Further since there are no non trivial inclusions, the ideal $\langle \bar{i} \rangle $ is $ \bar{i} $ and so $\Gamma$ on $\langle \bar{i} \rangle $ is an isomorphism. Hence $\Gamma$ is a local isomorphism.
\end{proof}
\begin{cor}\label{cormg}
$\Gamma$ as defined in equation \ref{eqngamma} is a cross-connection\index{cross-connection}.
\end{cor}
\begin{proof}
By Proposition \ref{prolociso}, $\Gamma$ is a local isomorphism. Since $\bar{\Lambda}$ is a groupoid, $M\Gamma(\bar{i}) = \bar{\Lambda}$ for every $\bar{i} \in \bar{I}$. And so for every $\bar{\lambda} \in \bar{\Lambda}$, there is a $\bar{i} \in \bar{I}$ (infact for every $\bar{i} \in \bar{I}$) such that $\bar{\lambda} \in M\Gamma(\bar{i})$. Hence $\Gamma$ is a cross-connection.
\end{proof}
\begin{thm}\label{thmstcss1}
Each cross-connection $\Gamma$ in a completely simple semigroup determines a $\Lambda\times I$ matrix $A$ with entries from the group $G$ such that $\Gamma(\bar{i}) = A_i$ where $A_i$ is the $i-$th column of A.
\end{thm}
\begin{proof}
Given a cross-connection $\Gamma : \bar{I} \to G^\Lambda/G$ such that $v\Gamma:\bar{i} \mapsto a_i  G $ where $a_i \in G^\Lambda$, let $A$ be an $\Lambda\times I$ matrix whose columns are $a_i$ for $i \in I$. Since $a_i \in G^\Lambda$, we have $a_i(\lambda) = a_{\lambda i} \in G$. Then $A$ is a $\Lambda\times I$ matrix with entries in $G$.  
\end{proof}
\begin{pro}\label{produal}
Similarly a dual cross-connection\index{cross-connection!dual} $\Delta : \bar{\Lambda} \to G^I/G$ can  be defined as follows.
\begin{equation}\label{eqndelta}
v\Delta:\bar{\lambda}  \mapsto G b_\lambda  \qquad \Delta:\: \rho(g) \mapsto \tau_g 
\end{equation} 
for any $b_\lambda \in G^I$. And this gives rise to a $\Lambda \times I$ matrix $B$.
\end{pro}
\begin{proof}
As in the proof of propositions \ref{prolociso}, $\Delta$ is a fully-faithful, inclusion preserving covariant functor. And since there are no non trivial inclusions, $\langle \bar{\lambda} \rangle = \bar{\lambda}$ and hence $\Delta_{\langle \bar{\lambda} \rangle}$ is an isomorphism. Also since $M\Gamma(\bar{\lambda}) = \bar{I}$ for every $\bar{\lambda} \in \bar{\Lambda}$, $\Delta$ is a cross-connection. Further defining the $\lambda$-th row $B_{\lambda}$ of the matrix $B$ as $B_{\lambda}\: =b_\lambda$, $B$ will be an $\Lambda \times I$ matrix with entries from the group $G$.
\end{proof}
We have seen that cross-connections in a completely simple semigroup correspond to matrices. Now we proceed to show that given a matrix $P$ we can define a cross-connection and its dual. We construct the cross-connection semigroup of this connection and show that it is isomorphic to the completely simple semigroup $\mathscr{M}[G;I,\Lambda;P]$. For that end, we need the following results.
\begin{pro}\label{procross}
Given a $\Lambda\times I$ matrix\index{sandwich matrix} $ P$ with entries from the group $G$, we define $\Gamma : \bar{I} \to G^\Lambda/G$ such that
\begin{equation}\label{eqngp}
v\Gamma:\bar{i}  \mapsto p_i  G \qquad \Gamma:\: \kappa(g) \mapsto \sigma_g 
\end{equation} 
where $p_i$ is the i-th column of the matrix $ P$. This gives a \index{cross-connection}cross-connection such that the functor defined by $\Delta : \bar{\Lambda} \to G^I/G$ as follows is the dual connection\index{cross-connection!dual}.
\begin{equation}\label{eqndp}
v\Delta:\bar{\lambda}  \mapsto G  \bar{p_\lambda}  \qquad \Delta:\: \rho(g) \mapsto \tau_g 
\end{equation} 
where $\bar{p_\lambda}$ is the $\lambda$-th row of $P$.
\end{pro}
\begin{proof}
The proof follows from the proof of Proposition \ref{prolociso}, Corollary \ref{cormg} and Proposition \ref{produal}.
\end{proof}
Given a cross-connection and a dual, there are bi-functors\index{functor!bi-functor} associated with each of them. The following lemma characterizes the bi-functors associated with the above cross-connections $\Gamma$ and $\Delta$.
\begin{lem}\label{lembif}
The bi-functors $\Gamma(-,-)$ and $\Delta(-,-)$ associated with the above cross-connection are defined as follows.\\
For $(\bar{\lambda},\bar{i}) \in \bar{\Lambda}\times\bar{I}$
\begin{equation}\label{bifg}
\Gamma(-,-):\bar{\Lambda}\times\bar{I} \to \mathbf{Set},\: \Gamma(\bar{\lambda},\bar{i}) = (p_i  G, \lambda) 
\end{equation}
and for $\rho({g_2}):\bar{\lambda_1}\to \bar{\lambda_2},\qquad\kappa({g_1}): \bar{i_1} \to \bar{i_2}$
$$\Gamma(\rho({g_2}),\kappa({g_1})): (p_{i_1}g, \lambda_1) \mapsto (p_{i_2} g_1 g g_2, \lambda_2).$$
Also for $(\bar{\lambda},\bar{i}) \in \bar{\Lambda}\times\bar{I}$
\begin{equation}\label{bifd}
\Delta(-,-): \bar{\Lambda}\times\bar{I} \to \mathbf{Set},\: \Delta(\bar{\lambda},\bar{i})  = (G \bar{p_\lambda}, i)
\end{equation}
and for $\rho({g_2}):\bar{\lambda_1}\to \bar{\lambda_2},\qquad\kappa({g_1}): \bar{i_1} \to \bar{i_2}$
$$\Delta(\rho({g_2}),\kappa({g_1})): (g \bar{p_{\lambda_1}}, i_1) \mapsto (g_1 g g_2 \bar{p_{\lambda_2}}, i_2).$$
\end{lem}
\begin{proof}
Recall from Remark \ref{bif} that given a cross-connection $\Gamma: \mathcal{D} \to N^\ast\mathcal{C}$, we have a unique bifunctor $\Gamma(-,-) : \mathcal{C}\times\mathcal{D} \to \bf{Set}$ defined by $ \Gamma(c,d) = \Gamma(d)(c) $ and $ \Gamma(f,g) = (\Gamma(d)(f))(\Gamma(g)(c')) = (\Gamma(g)(c))(\Gamma(d')(f))$ for all $(c,d) \in v\mathcal{C}\times v\mathcal{D}$ and $(f,g):(c,d) \to (c',d')$. So given the cross-connection $\Gamma : \bar{I} \to G^\Lambda/G$ as defined in equation \ref{eqngp}, $\Gamma(\bar{\lambda},\bar{i}) = \Gamma(\bar{i})(\bar{\lambda}) = (p_i  G, \lambda)$.\\
And given $(\rho({g_2}),\kappa({g_1})): ( \bar{\lambda_1},\bar{i_1}) \to (\bar{\lambda_2},\bar{i_2})$,  $$\Gamma(\rho({g_2}),\kappa({g_1})) =  (\Gamma(\bar{i_1})(\rho(g_2))) (\Gamma(\kappa({g_1}))(\bar{\lambda_2})) = (p_{i_1} G (\rho({g_2}))) (\sigma_{g_1}(\bar{\lambda_2})).$$
And so for $\gamma =(p_{i_1} g,\lambda_1) \in \Gamma(\bar{\lambda_1},\bar{i_1})$, 
$$((p_{i_1} g,\lambda_1))\:(p_{i_1} G (\rho({g_2}))) (\sigma_{g_1}(\bar{\lambda_2})) = (p_{i_1} g g_2,\lambda_2) \sigma_{g_1}(\bar{\lambda_2}) = (p_{i_2} g_1 g g_2, \lambda_2).$$
Also observe that $((p_{i_1} g,\lambda_1))(\Gamma(\kappa({g_1}))(\bar{\lambda_1}))(\Gamma(\bar{i_2})(\rho({g_2})))\: = (p_{i_2} g_1 g g_2, \lambda_2).$ \\
Hence $\Gamma(\rho({g_2}),\kappa({g_1})): (p_{i_1} g, \lambda_1) \mapsto (p_{i_2} g_1 g g_2, \lambda_2)$.\\
Similarly corresponding to $\Delta: \mathcal{C} \to N^\ast\mathcal{D}$, we have $\Delta(-,-) : \mathcal{C}\times\mathcal{D} \to \bf{Set}$ defined by $\Delta(c,d) = \Delta(c)(d)$ and $\Delta(f,g) = (\Delta(c)(g))(\Delta(f)(d')) = (\Delta(f)(d))(\Delta(c')(g)) $ for all $(c,d) \in v\mathcal{C}\times v\mathcal{D}$ and $(f,g):(c,d) \to (c',d')$. So given the dual cross-connection $\Delta : \bar{\Lambda} \to G^I/G$ as defined in equation \ref{eqndelta}, 
$$\Delta(\bar{\lambda},\bar{i}) = \Delta(\bar{\lambda})(\bar{i}) = (G  \bar{p_\lambda}, i).$$
And given $(\rho({g_2}),\kappa({g_1})): ( \bar{\lambda_1},\bar{i_1}) \to (\bar{\lambda_2},\bar{i_2})$,  $$\Delta(\rho({g_2}),\kappa({g_1})) =  (\Delta(\bar{\lambda_1})(\kappa({g_1})))(\Delta(\rho(g_2))(\bar{i_2})) = (G \bar{p_{\lambda_1}} (\kappa({g_1})))( \tau_{g_2}(\bar{i_2})).$$
And so for $\delta =(g \bar{p_{\lambda_1}},i_1) \in \Delta(\bar{\lambda_1},\bar{i_1})$, 
$$((g \bar{p_{\lambda_1}},i_1))\:(G \bar{p_{\lambda_1}} (\kappa({g_1})))( \tau_{g_2}(\bar{i_2})) = (g_1 g \bar{p_{\lambda_1}}, i_2) \tau_{g_2}(\bar{i_2}) = (g_1 g g_2 \bar{p_{\lambda_2}}, i_2).$$
Also $((g \bar{p_{\lambda_1}},i_1)) (\Delta(\rho({g_2}))(\bar{i_1}))(\Delta(\bar{\lambda_2})(\kappa({g_1}))\: = (g_1 g g_2 \bar{p_{\lambda_2}}, i_2)$. Hence the lemma.
\end{proof}
Given a cross-connection, we can associate two regular semigroups $U\Gamma$ and $U\Delta$ with them. The following lemma describes these semigroups.
\begin{lem}\label{lemug}
Let $\Gamma$ be the cross-connection of $\mathcal{R}(S)$ with $\mathcal{L}(S)$ of the completely simple semigroup $S$ as defined in equation \ref{eqngp}. Then
\begin{subequations}
\begin{align}
U\Gamma & =  \{ (\bar{\gamma},\lambda) \in G^\Lambda \times \Lambda \::\: \bar{\gamma} = p_{i}g \text{ for some } g \in G \text{ and } i \in I;\: \lambda \in \Lambda\} \\
U\Delta & =  \{ (\bar{\delta},i) \in G^I \times I \::\: \bar{\delta} = g\bar{p_{\lambda }} \text{ for some } g \in G \text{ and } \lambda \in \Lambda; \: i \in I \}.
\end{align}
\end{subequations}
And $U\Gamma$ and $U\Delta$ form regular subsemigroups of $G^\Lambda \ltimes \Lambda$ and $(G^I \ltimes I)^{\text{op}}$ respectively.
\end{lem}
\begin{proof}
Recall that given a cross-connection $\Gamma$ of $\mathcal{D}$ with $\mathcal{C}$, $
U\Gamma =  \bigcup\: \{ \quad \Gamma(c,d) : (c,d) \in v\mathcal{C} \times v\mathcal{D} \} $ and $
U\Delta =  \bigcup\: \{ \quad \Delta(c,d) : (c,d) \in v\mathcal{C} \times v\mathcal{D} \}$.\\
Since $\Gamma(\bar{\lambda},\bar{i}) = (p_i   G, \lambda) = \{(p_i   g, \lambda) \::\: g \in G\} $,\\
$U\Gamma  = \bigcup\{(p_{  i} g, \lambda) \::\: g \in G \text{ and } (\lambda,i) \in \Lambda \times I \} =  \{ (p_{  i}g,\lambda)\::\:g \in G, i \in I, \lambda \in \Lambda\}$. Hence 
$$U\Gamma =  \{ (\bar{\gamma},\lambda) \in G^\Lambda \ltimes \Lambda \::\: \bar{\gamma} = p_{  i}g \text{ for some } g \in G \text{ and } i \in I;\: \lambda \in \Lambda\} $$
Now given $a=(p_{  i_a}g_a,\lambda_a), b=(p_{  i_b}g_b,\lambda_b) \in U\Gamma$,\\
$(p_{  i_a}g_a,\lambda_a).(p_{  i_b}g_b,\lambda_b) = (p_{  i_a}g_ap_{\lambda_a i_b}g_b,\lambda_b)$. Since $g_ap_{\lambda_a i_b}g_b =g' \in G$, we have $a.b = (p_{  i_a}g',\lambda_b) \in U\Gamma$. So $U\Gamma$ is a subsemigroup of $G^\Lambda \ltimes \Lambda$.\\ 
Also given $a=(p_{  i_a}g_a,\lambda_a)$, if we take $b= (p_{  i_a}(p_{\lambda_a i_a}g_ap_{\lambda_a i_a})^{-1},\lambda_a)$\\
then $aba = (p_{  i_a}g_a,\lambda_a)(p_{  i_a}(p_{\lambda_a i_a}g_ap_{\lambda_a i_a})^{-1},\lambda_a)(p_{  i_a}g_a,\lambda_a)$\\ $=  (p_{  i_a}g_ap_{\lambda_a i_a}(p_{\lambda_a i_a}g_ap_{\lambda_a i_a})^{-1},\lambda_a)(p_{  i_a}g_a,\lambda_a) $ 
$= (p_{  i_a}g_ap_{\lambda_a i_a}(p_{\lambda_a i_a}g_ap_{\lambda_a i_a})^{-1}p_{\lambda_a i_a}g_a,\lambda_a)$ \\$ = (p_{  i_a}g_ap_{\lambda_a i_a}p_{\lambda_a i_a}^{-1}g_a^{-1}p_{\lambda_a i_a}^{-1}p_{\lambda_a i_a}g_a,\lambda_a) = (p_{  i_a}g_a(p_{\lambda_a i_a}p_{\lambda_a i_a}^{-1})g_a^{-1}(p_{\lambda_a i_a}^{-1}p_{\lambda_a i_a})g_a,\lambda_a)$ \\$  = (p_{  i_a}g_ag_a^{-1}g_a,\lambda_a) = (p_{  i_a}(g_ag_a^{-1})g_a,\lambda_a) = (p_{  i_a}g_a,\lambda_a) = a $.\\
Hence $a^{-1} = (p_{  i_a}(p_{\lambda_a i_a}g_ap_{\lambda_a i_a})^{-1},\lambda_a)$ and $U\Gamma$ is regular.\\
Similarly since $\Delta(\bar{\lambda},\bar{i}) = ( G   \bar{p_\lambda}, i) = \{(g   \bar{p_\lambda}, i) \::\: g \in G\} $, 
$U\Delta  = \bigcup\{(g\bar{p_{\lambda}}, i) \::\: g \in G \text{ and } (\lambda,i) \in \Lambda \times I \}$
$ =  \{ (g \bar{p_{\lambda}}, i)\::\:g \in G, i \in I, \lambda \in \Lambda \}$. Hence 
$$U\Delta  =  \{ (\bar{\delta},i) \in G^I \ltimes I \::\: \bar{\delta} = g\bar{p_{\lambda}} \text{ for some } g \in G \text{ and } \lambda \in \Lambda; \: i \in I \}.$$
Now given $a=(g_a\bar{p_{\lambda_a}},i_a), b=(g_b\bar{p_{\lambda_b}},i_b) \in U\Delta$, taking $g_bp_{\lambda_b i_a}g_a =g' \in G$, we have $a.b = (g'\bar{p_{\lambda_a}},i_b) \in U\Delta$ and $U\Delta$ is a subgroup of $(G^I \ltimes I)^{\text{op}}$.\\ 
Also given $a=(g_a\bar{p_{\lambda_a}},i_a)$, if we take $b= ((p_{\lambda_a i_a}g_ap_{\lambda_a i_a})^{-1}\bar{p_{\lambda_a}},i_a)$, then $aba = a $ and hence $U\Delta$ is regular.\\
Hence the lemma.
\end{proof}
For a cross-connection $\Gamma$ and the associated bi-functors $\Gamma(-,-) $ and $ \Delta(-,-)$, we can associate a natural isomorphism between them called the \index{duality}duality $\chi_\Gamma $. The following lemma describes the duality associated with the cross-connection $\Gamma$.
\begin{lem}\label{lemduality}
Given the cross-connection $\Gamma$ as defined in equation \ref{eqngp}, the duality $\chi_\Gamma : \Gamma(-,-) \to \Delta(-,-)$ associated with $\Gamma$ is given by
\begin{equation}\label{eqnduality}
\chi_\Gamma(\bar{\lambda},\bar{i}) : (p_i   g ,\lambda) \mapsto (g \bar{p_\lambda} , i)
\end{equation}
\end{lem}
\begin{proof}
To show that $\chi_\Gamma$ is the duality associated with $S$, we firstly show that the map $\chi_\Gamma(\bar{\lambda},\bar{i})$ is a bijection of the set $\Gamma(\bar{\lambda},\bar{i})$ onto the set $\Delta(\bar{\lambda},\bar{i})$. But since $\Gamma(\bar{\lambda},\bar{i})$ is the set $(p_i   G ,\lambda) = \{(p_i   g ,\lambda) \::\: g \in G \}$, $|\Gamma(\bar{\lambda},\bar{i})| = |G|$. Similarly $\Delta(\bar{\lambda},\bar{i})$ is the set $(g \bar{p_\lambda} , i) = \{(g \bar{p_\lambda} , i) \::\: g \in G \}$ and so $|\Delta(\bar{\lambda},\bar{i})| = |G|$. Hence $|\Gamma(\bar{\lambda},\bar{i})| = |\Delta(\bar{\lambda},\bar{i})| $ and hence the map $\chi_\Gamma(\bar{\lambda},\bar{i})$ is a bijection.\\
Now we show that $\chi_\Gamma$ is a natural isomorphism. Let $(\rho({g_2}),\kappa({g_1})) \in \mathcal{L}(S) \times \mathcal{R}(S)$ with $\rho({g_2}) : \bar{\lambda_1} \to \bar{\lambda_2}$ and $\kappa({g_1}) : \bar{i_1} \to \bar{i_2}$. Then for $a = (p_{i_1}g_a ,\lambda_1) \in \Gamma(\bar{\lambda_1},\bar{i_1})$, 
$$((p_{i_1}g_a ,\lambda_1))\chi_\Gamma(\bar{\lambda_1},\bar{i_1})\Delta(\rho(g_2),\kappa(g_1)) = ((g_a\bar{p_{\lambda_1}} ,i_1))\Delta(\rho(g_2),\kappa(g_1))= (g_1g_ag_2\bar{p_{\lambda_2}} ,i_2).$$
Also 
$$((p_{i_1}g_a ,\lambda_1))\Gamma(\rho(g_2),\kappa(g_1))\chi_\Gamma(\bar{\lambda_2},\bar{i_2}) = 
((p_{i_2}g_1g_ag_2 ,\lambda_2))\chi_\Gamma(\bar{\lambda_2},\bar{i_2}) = (g_1g_ag_2\bar{p_{\lambda_2}} ,i_2).$$
Hence $(a)\chi_\Gamma(\bar{\lambda_1},\bar{i_1})\Delta(\rho(g_2),\kappa(g_1)) = (a)\Gamma(\rho(g_2),\kappa(g_1))\chi_\Gamma(\bar{\lambda_2},\bar{i_2}) \qquad \forall a \in \Gamma(\bar{\lambda_1},\bar{i_1})$ and so $\chi_\Gamma$ is a natural isomorphism. And $\chi_\Gamma$ is the duality associated with $S$.
\end{proof}
For a cross-connection $\Gamma$, we consider the pairs of linked cones $(\gamma,\delta) \in U\Gamma\times U\Delta$ inorder to construct the cross-connection semigroup. The following lemma gives the characterization of the linked cones\index{cone!linked} in terms of the cross-connection.
\begin{pro}\label{prolinked}
$\gamma =(\bar{\gamma},\lambda) \in U\Gamma$ is linked to $\delta =(\bar{\delta}, i) \in U\Delta \: $ if and only if $\bar{\gamma} \in \Gamma(\bar{i}) $, $\bar{\delta} \in \Delta(\bar{\lambda}) $ and $({p_i} ^{-1}\bar{\gamma})_1 = (\bar{\delta}{\bar{p_\lambda}} ^{-1})_1$ where $p_i$ is the i-th column and $\bar{p_\lambda}$ is the $\lambda$-th row of $ P$.
\end{pro}
\begin{proof}
Firstly observe that ${p_i} ^{-1} = ({p_{\lambda_1i}} ^{-1},{p_{\lambda_2i}} ^{-1}, ...)$ and ${\bar{p_\lambda}} ^{-1} = ({p_{\lambda i_1}} ^{-1},{p_{\lambda i_2}} ^{-1}, ...)$; and hence ${p_i} ^{-1}\bar{\gamma} \in G^\Lambda$ and $\bar{\delta}{\bar{p_\lambda}} ^{-1} \in G^I$. Thus $({p_i} ^{-1}\bar{\gamma})_1 $ and $ (\bar{\delta}{\bar{p_\lambda}} ^{-1})_1$ denote ${p_i} ^{-1}\bar{\gamma}(\lambda_1)$ and $ \bar{\delta}{\bar{p_\lambda}} ^{-1}(\lambda_1)$ respectively.\\
Now recall that $\gamma \in U\Gamma$ is linked to $\delta \in U\Delta$ if there is a $(\bar{\lambda},\bar{i}) \in \bar{\Lambda} \times \bar{I}$ such that $\gamma \in \Gamma(\bar{\lambda},\bar{i})$ and $ \delta = \chi_\Gamma(\bar{\lambda},\bar{i})(\gamma)$.
If $\gamma =( p_{ i}g, \lambda )$ is linked to $\delta $, then there is a $(\bar{\lambda },\bar{i }) \in \bar{\Lambda} \times \bar{I}$ such that $\gamma \in \Gamma(\bar{\lambda },\bar{i })$ and $ \delta = \chi_\Gamma(\bar{\lambda },\bar{i })(\gamma)$. Then $\gamma = ( p_{ i }g , \lambda ) \in \Gamma(\bar{i })(\bar{\lambda }) $ and $\Gamma(\bar{i }) =  p_{ i }G$. Thus $  \bar{\gamma} = p_{ i }g \in \Gamma(\bar{i })$. Also since $ \delta = \chi_\Gamma(\bar{\lambda },\bar{i })(\gamma) = \chi_\Gamma(\bar{\lambda },\bar{i })(p_{ i }g , \lambda ) = (g  p_{\lambda } , i ) \in \Delta(\bar{\lambda })(\bar{i }) $ and $\Delta(\bar{\lambda }) = G  p_{\lambda } $,  $\bar{\delta} = g  p_{\lambda } \in \Delta(\bar{\lambda })$. Now ${p_i} ^{-1}\bar{\gamma} = {p_i} ^{-1}p_{ i}g = ({p_{\lambda_1i}} ^{-1}p_{\lambda_1 i}g,{p_{\lambda_2i}} ^{-1}p_{\lambda_2 i}g,...) = (g,g,...)$; and hence $({p_i} ^{-1}\bar{\gamma})_1 =g$. Similarly since $\bar{\delta}{\bar{p_\lambda}} ^{-1} = (g \bar{p_{\lambda }}{\bar{p_\lambda}} ^{-1}) = (g,g,...)$, $ (\bar{\delta}{\bar{p_\lambda}} ^{-1})_1 = g$ and hence $({p_i} ^{-1}\bar{\gamma})_1 = (\bar{\delta}{\bar{p_\lambda}} ^{-1})_1$.\\
Conversely if $\bar{\gamma} \in \Gamma(\bar{i}) $, $\bar{\delta} \in \Delta(\bar{\lambda}) $ and $({p_i} ^{-1}\bar{\gamma})_1 = (\bar{\delta}{\bar{p_\lambda}} ^{-1})_1$, then since $\bar{\gamma} \in \Gamma(\bar{i}) $, $\gamma = (\bar{\gamma},\lambda) \in \Gamma(\bar{i})(\bar{\lambda}) = \Gamma(\bar{\lambda},\bar{i})$. Also since $\bar{\delta} \in \Delta(\bar{\lambda})$, $\delta = (\bar{\delta},i) \in \Delta(\lambda)(i) = \Delta(\bar{\lambda},\bar{i})$; and hence $\delta = (g_1\bar{p_{\lambda}},i)$ for some $g_1 \in G$. Now if $\gamma =( p_{ i}g, \lambda )$, then $\chi_\Gamma(\bar{\lambda },\bar{i })(\gamma) = \chi_\Gamma(\bar{\lambda },\bar{i })( p_{ i}g, \lambda ) = (g\bar{p_{\lambda}},i)$. Now since $({p_i} ^{-1}\bar{\gamma})_1 = (\bar{\delta}{\bar{p_\lambda}} ^{-1})_1$, $(\bar{\delta})_1 = p_{1i}^{-1} (\bar{\gamma})_1{p_{\lambda 1}} = p_{1i}^{-1} (p_{1i}g){p_{\lambda 1}} = g{p_{\lambda 1}}$. But since $(\bar{\delta})_1 = (g_1p_{\lambda 1})$, we have $g_1p_{\lambda 1} = g{p_{\lambda 1}}$, and hence $g_1 = g$. Thus $\delta = (g\bar{p_{\lambda}},i) = \chi_\Gamma(\bar{\lambda },\bar{i })(\gamma)$ and $\gamma$ is linked to $\delta$.
\end{proof}
Now we proceed to give an alternate description of the linked cones in terms of the components which will lead us to the representation of the completely simple semigroup as a cross-connection semigroup.\index{cross-connection!semigroup}
\begin{lem}\label{lemlinkpair}
$\gamma =(p_{  i_1}g_1,\lambda_1) \in U\Gamma$ is linked to $\delta =(g_2\bar{p_{\lambda_2}}, i_2) \in U\Delta \: $ if and only if $ i_1 =i_2, \lambda_1 = \lambda_2$ and $g_1 =g_2$. Hence a linked cone pair will be of the form $((\bar{\gamma},\lambda),(\bar{\delta}, i))$ for $i \in I, \lambda \in \Lambda $ and $\bar{\gamma} = p_{  i}g$ and $\bar{\delta} = g\bar{p_{\lambda}}$ for some $g \in G$.
\end{lem}
\begin{proof}
As above, $\gamma \in U\Gamma$ is linked to $\delta \in U\Delta$ if there is a $(\bar{\lambda},\bar{i}) \in \bar{\Lambda} \times \bar{I}$ such that $\gamma \in \Gamma(\bar{\lambda},\bar{i})$ and $ \delta = \chi_\Gamma(\bar{\lambda},\bar{i})(\gamma)$. Also recall that $\gamma \in U\Gamma$ is linked to $\delta \in U\Delta$ with respect to $\Gamma$ if and only if $\delta$ is linked to $\gamma$ with respect to $\Delta$ .ie there is a $(\bar{\lambda},\bar{i}) \in \bar{\Lambda} \times \bar{I}$ such that $\delta \in \Delta(\bar{\lambda},\bar{i})$ and $ \gamma = \chi_{\Delta}(\bar{\lambda},\bar{i})(\delta)$ where $\chi_{\Delta} = \chi_{\Gamma^{-1}} : \Delta(-,-) \to \Gamma(-,-)$.\\
Now firstly, if $i_1 =i_2, \lambda_1 = \lambda_2$ and $g_1 =g_2$, then $\chi_\Gamma(\bar{\lambda_1},\bar{i_1})(\gamma) = \chi_\Gamma(\bar{\lambda_1},\bar{i_1})(p_{  i_1}g_1,\lambda_1) = (g_1\bar{p_{\lambda_1}}, i_1) =$ $(g_2\bar{p_{\lambda_2}}, i_2) = \delta$ and hence $\gamma$ is linked to $\delta$.\\
Conversely if $\gamma$ is linked to $\delta$, then $\chi_\Gamma(\bar{\lambda_1},\bar{i_1})(\gamma) = \chi_\Gamma(\bar{\lambda_1},\bar{i_1})(p_{  i_1}g_1,\lambda_1) = (g_1\bar{p_{\lambda_1}}, i_1) = \delta = (g_2\bar{p_{\lambda_2}}, i_2)$. .ie $(g_1\bar{p_{\lambda_1}}, i_1) = (g_2\bar{p_{\lambda_2}}, i_2)$. And since $\delta$ will be linked to $\gamma$, we also have $(p_{  i_1}g_1,\lambda_1) = (p_{  i_2}g_2,\lambda_2)$. \\
Comparing the equations, we have $i_1 =i_2$, $\lambda_1 = \lambda_2$ and $p_{\lambda_k i_1}g_1 = p_{\lambda_k i_2}g_2$ for every $k \in \Lambda$. Since $i_1 =i_2$ and in particular for $k=1$, we have $p_{\lambda_1 i_1}g_1 = p_{\lambda_1 i_1}g_2$. So $g_1 =g_2$ and hence the proof.\\
The second statement is now straight forward.
\end{proof}
Earlier we showed that a cross-connection determines a matrix. Now we are in a position to show that given a matrix $P$, we can define a cross-connection and use it to construct a cross-connection semigroup which will be isomorphic to the completely simple semigroup $\mathscr{M}[G;I,\Lambda;P]$. 
\begin{thm}\label{thmstcss2}
Given a $\Lambda \times I$ matrix $P$ with entries from the group $G$, we can define a cross-connection $\Gamma$. And the cross-connection semigroup $\tilde{S}\Gamma$ is isomorphic to the completely simple semigroup $S = \mathscr{M}[G;I,\Lambda;P]$.   
\end{thm}
\begin{proof}
Given a $\Lambda \times I$ matrix $P$ with entries from the group $G$, define $\Gamma$ as defined in equations \ref{eqngp} and \ref{eqndp}. By Proposition \ref{procross}, it is a cross-connection.
Now recall that the cross-connection semigroup $\tilde{S}\Gamma $ of linked cone pairs is given by 
$$ \tilde{S}\Gamma = \:\{\: (\gamma,\delta) \in U\Gamma\times U\Delta : (\gamma,\delta) \text{ is linked }\:\} $$ 
with the binary operation defined by 
$$ (\gamma , \delta) \circ ( \gamma' , \delta') = (\gamma . \gamma' , \delta' . \delta)    $$
for all $(\gamma,\delta),( \gamma' , \delta') \in \tilde{S}\Gamma$. 
Firstly from Lemma \ref{lemlinkpair}; a linked cone pair will be of the form $((p_{  i}g,\lambda),(gp_{\lambda  }, i))$ for $g \in G, i \in I, \lambda \in \Lambda $. Thus  
\begin{equation}\label{cxnsg}
\tilde{S}\Gamma = \{((p_{  i}g,\lambda),(g\bar{p_{\lambda}}, i)) \::\: g \in G, i\in I, \lambda \in \Lambda \}
\end{equation}
with the binary operation $\circ$ as follows.
\begin{equation} \label{cxnsgo}
\begin{split}
((p_{  i_1}g_1,\lambda_1),(g_1\bar{p_{\lambda_1}}, i_1))\circ((p_{  i_2}g_2,\lambda_2),(g_2\bar{p_{\lambda_2}}, i_2))& =\\
 ((p_{  i_1}g_1p_{\lambda_1 i_2}&g_2,\lambda_2),(g_1p_{\lambda_1 i_2}g_2\bar{p_{\lambda_2}}, i_1))
\end{split}
\end{equation}
Observe here that the multiplication in the second component is taken in the opposite order .ie in $U\Delta^\text{op}$. Now we will show that $\tilde{S}\Gamma$ is isomorphic to the completely simple semigroup $S = \mathscr{M}[G;I,\Lambda;P]$.\\
For that define a map $\varphi : S \to \tilde{S}\Gamma$ as follows.
\begin{equation}\label{eqnphi}
(g,i,\lambda)\varphi = ((p_{  i}g,\lambda),(g\bar{p_{\lambda}}, i))
\end{equation}
Clearly $\varphi$ is well-defined and is onto. The fact that $\varphi$ is a homomorphism follows from the proof of Proposition \ref{proprinci} and its dual; by observing that $((p_{  i}g,\lambda),(g\bar{p_{\lambda}}, i)) = (\rho^{(g,i,\lambda)},\lambda^{(g,i,\lambda)})$. Now to see that it is injective, suppose if $(g_1,i_1,\lambda_1)\varphi = (g_2,i_2,\lambda_2)\varphi$. Then $((p_{  i_1}g_1,\lambda_1),(g_1\bar{p_{\lambda_1}}, i_1)) = ((p_{  i_2}g_2,\lambda_2),(g_2\bar{p_{\lambda_2}}, i_2))$. Comparing the vertices, we see that $i_1 =i_2$ and $\lambda_1 =\lambda_2$. We also have $p_{  i_1}g_1 = p_{  i_2}g_2$ .ie $p_{ \lambda_k i_1}g_1 = p_{\lambda_k  i_2}g_2$ for every $k$, and in particular for $k=1$ .ie $p_{\lambda_1 i_1}g_1 = p_{\lambda_1 i_2}g_2$. And since $i_1 =i_2$, we have $p_{\lambda_1 i_1}g_1 = p_{\lambda_1 i_1}g_2 \: \implies g_1 =g_2$. Thus $(g_1,i_1,\lambda_1)=(g_2,i_2,\lambda_2) $ and hence $\varphi$ is injective. Thus $ \tilde{S}\Gamma$ is isomorphic to the completely simple semigroup $S = \mathscr{M}[G;I,\Lambda;P]$. 
\end{proof}
Thus the semigroup $\tilde{S}\Gamma$ as defined in equation \ref{cxnsg} is a representation of the completely simple semigroup $S$, wherein each element $(g,i,\lambda)$ of the semigroup is represented as a pair of normal cones $((p_{  i}g,\lambda),(gp_{\lambda  }, i))$ and multiplication is a semi-direct product multiplication on each component (see equation \ref{cxnsgo}). We conclude by giving a special case of the result using which one can construct a rectangular band of groups.
\begin{cor}\label{RBG}
We can define a cross-connection $\Gamma : \bar{I} \to G^\Lambda/G$ such that
\begin{equation}\label{eqnrbg}
v\Gamma:\bar{i}  \mapsto  G \qquad \Gamma:\: \kappa(g) \mapsto \sigma_g 
\end{equation} 
with the dual connection defined by $\Delta : \bar{\Lambda} \to G^I/G$ 
\begin{equation}\label{eqndrbg}
v\Delta:\bar{\lambda}  \mapsto G   \qquad \Delta:\: \rho(g) \mapsto \tau_g. 
\end{equation} 
Then 
$$U\Gamma =  \{ (\bar{g},\lambda)\::\: g \in G, \lambda \in \Lambda\} \text{   and   }
U\Delta  =  \{ (\bar{g}, i)\::\:g \in G, i \in I\}.$$
where $\bar{g}= (g,g,g,...)$. Then the duality $\chi_\Gamma : \Gamma(-,-) \to \Delta(-,-)$ associated with $S$ is given by
\begin{equation}\label{eqndualityRBG}
\chi_\Gamma(\bar{\lambda},\bar{i}) : ( \bar{g} ,\lambda) \mapsto (\bar{g}  , i)
\end{equation}
And consequently $\gamma =(\bar{g_1},\lambda) \in U\Gamma$ is linked to $\delta =(\bar{g_2}, i) \in U\Delta \: \iff g_1 =g_2$.\\ 
Further $\tilde{S}\Gamma$ for this cross-connection is given by 
\begin{equation}\label{cxnsgrbg}
\tilde{S}\Gamma = \{((\bar{g},\lambda),(\bar{g}, i)) \::\: g \in G, i \in I, \lambda \in \Lambda \}
\end{equation}
with the binary operation $\circ$ 
\begin{equation} \label{cxnsgorbg}
((\bar{g_1},\lambda_1),(\bar{g_1}, i_1))\circ((\bar{g_2},\lambda_2),(\bar{g_2}, i_2)) \:= \:
 ((\overline{g_1g_2},\lambda_2),(\overline{g_1g_2}, i_1)).
\end{equation}
Then $\tilde{S}\Gamma$ is isomorphic to the rectangular band of groups $S = G \times I \times \Lambda$.
\end{cor}
\begin{proof}
The corollary follows directly from the earlier results on completely simple semigroup $ \mathscr{M}[G;I,\Lambda;P]$ by taking the sandwich matrix $P$ as $p_{\lambda i} = e $ for all $\lambda \in \Lambda$ and $i \in I $ where $e$ is the identity element of the group $G$.
\end{proof}

\end{document}